\documentclass[reqno, 12pt]{amsart}

\usepackage{mathrsfs}
\usepackage{amscd}
\usepackage{amsmath}
\usepackage{latexsym}
\usepackage{amsfonts}
\usepackage{amssymb}
\usepackage{amsthm}
\usepackage{graphicx}
\usepackage{hyperref}
\usepackage{makecell}
\usepackage{array}
\usepackage{booktabs}
\usepackage{multirow}
\usepackage{color,xcolor}

\newtheorem{theorem}{Theorem}[section]
\newtheorem{proposition}[theorem]{Proposition}
\newtheorem{corollary}[theorem]{Corollary}
\newtheorem{lemma}[theorem]{Lemma}
\newtheorem{assumption}[theorem]{Assumption}
\newtheorem{definition}[theorem]{Definition}

\newtheorem{remark}[theorem]{Remark}

\numberwithin{equation}{section}

\title{Stability estimate for the discrete Calder\'on problem from  partial data}

 \author[Xiaomeng Zhao]{Xiaomeng Zhao}
\address{KLAS, School of Mathematics and Statistics, Northeast Normal University,
Changchun, Jilin, 130024, China}
\email{zhaoxm600@nenu.edu.cn}

 \author[Ganghua Yuan]{Ganghua Yuan}
\address{KLAS, School of Mathematics and Statistics, Northeast Normal University,
Changchun, Jilin, 130024, China}
\email{yuangh925@nenu.edu.cn}


\begin{document}

\begin{abstract}

In this paper, we focus on the analysis of discrete versions of the Calder\'on problem with partial boundary data in dimension $d\geq 3$. In particular, we establish logarithmic stability estimates for the discrete Calder\'on problem on an arbitrarily small portion of the boundary under suitable a priori bounds.  For this end, we will use CGO solutions and derive a new discrete Carleman estimate and a key unique continuation estimate. Unlike the continuous case, we use a new strategy inspired by \cite{robbiano1995fonction} to prove the key discrete unique continuation estimate by utilizing  the new Carleman estimate with boundary observations for a discrete Laplace operator.

\end{abstract}

\maketitle

\section{Introduction}
In this work, we are interested in a discrete version of the Calder\'on problem proposed by Calder\'on  \cite{calderon2006inverse} in 1980. He asked if it is possible to determine the electrical conductivity of a
medium by making voltage and current measurements at the boundary of the medium. This
inverse problem is known as the Calder\'on problem and has been the origin of many interesting
works. An excellent survey of the history of this problem and its developments can be found
in \cite{uhlmann201230}.

To be more precise, let ${\mathcal G}$ be a smooth bounded domain of $\mathbb R^d$. Let us denote by $n=n(x)$ the outward normal vector at $x\in\partial{\mathcal G}$. Given ${\mathcal\kappa}$ a positive conductivity function, we define the Dirichlet-to-Neumann (DtN) map associated with the conductivity problem
\begin{align}
	\Lambda({\mathcal\kappa}):H^{\frac{1}{2}}(\partial{\mathcal G})\rightarrow H^{-\frac{1}{2}}(\partial{\mathcal G}),\quad \Lambda(\mathcal\kappa):g\mapsto{\mathcal\kappa}\nabla u\cdot n,
\end{align}
where $u$ is the unique solution to the elliptic problem
\begin{align}\label{e12}
	{\rm div}({\mathcal\kappa}\nabla u)=0~{\rm in}~{\mathcal G}\quad{\rm and}\quad u=g~{\rm on}~\partial{\mathcal G}.
\end{align}
The Calder\'on problem is the following: given $\Lambda({\mathcal\kappa})$, can we retrieve the conductivity ${\mathcal\kappa}$? Such a question is considered as the archetype of inverse problems and was
originally motivated by oil prospection \cite{calderon2006inverse} but has application in EIT, in the medical imaging
method \cite{cheney1999electrical} and in several other inverse problems as detection of leaks in buried pipes \cite{jordana2001electrical}.

On  the other hand, it is well known that the Liouville transform allows rewriting (\ref{e12}) in the following way:
\begin{align}\label{e14}
	-{\mathcal\kappa}^{-\frac{1}{2}}{\rm div}\left({\mathcal\kappa}\nabla({\mathcal\kappa}^{-\frac{1}{2}}v)\right)=-\Delta v+qv=0,\quad {\rm where}~q=\frac{\Delta\sqrt{{\mathcal\kappa}}}{\sqrt{{\mathcal\kappa}}}.
\end{align}

Thus, Calder\'on's problem can be considered as: given $q$ a regular potential function, we define the DtN map in the following way:
\begin{align}
	\Lambda(q):H^{\frac{1}{2}}(\partial{\mathcal G})\rightarrow H^{-\frac{1}{2}}(\partial{\mathcal G}),\quad \Lambda(q):g\mapsto\nabla u\cdot n,
\end{align}
for a prescribed voltage $g$ on $\partial{\mathcal G}$, where $u$ is the unique solution of the problem
\begin{align}\label{l15}
	-\Delta u+qu=0~{\rm in}~{\mathcal G}\quad{\rm and}\quad u=g~{\rm on}~\partial{\mathcal G}.
\end{align}
Hence, the Calder\'on problem consists in recovering the potential $q$ from the DtN map in (\ref{e14}).

Focusing on $d\geq 3$, there are many references in the continuous case. And there are subcases that consider measurements on the whole boundary $\partial{\mathcal G}$ (full data) (see e.g., \cite{sylvester1987global})  or on a nonempty open subset of $\partial{\mathcal G}$ (partial data) (see e.g., \cite{bukhgeim2002recovering}, \cite{isakov2007uniqueness}, \cite{sjostrand2007calderon}, \cite{krupchyk2019stability}, \cite{2023Stability}). 
We can see the survey \cite{kenig2014recent} for more references.

 The discrete version of these inverse problems are interesting themselves \cite{horvarth2016}.
 The discrete version of the Calder\'on problem is related to resistor networks, where we have to determine the conductivities for an electrical network from the knowledge of the electrical potential and current in some nodes. The mathematical model for this is a weighted graph and a function defined on the vertices, where the weights of the edges are related to the conductivity of the edges, and the function is the electrical potential.
The discrete version of the Calder\'on problem has been studied in \cite{ervedoza2011uniform} and \cite{lecaros2023discrete}.
In \cite{ervedoza2011uniform}, a discrete stability estimate is obtained by using full boundary data.
	For the partial boundary data case, \cite{lecaros2023discrete} prove a stability result for the discrete Calder\'on problem in a cube in $\mathbb R^d, d\geq3$,  when measurements are taken on the boundary  except for one side of the cube. In this paper, we prove a stability result for the discrete Calder\'on problem with measurements  taken just on arbitrarily small portions of one side of the cube in $\mathbb R^d, d\geq3$.

This work is motivated by \cite{ervedoza2011uniform, krupchyk2019stability, lecaros2023discrete}.  To study the discrete Calder\'on problems, we first derive a Carleman estimate with boundary observation and a key unique continuation result for the discrete equation. Then,
we use the discrete CGO  solution (constructed in \cite{ervedoza2011uniform}) to bound the Fourier transform of the difference of the potentials $q_1-q_2$. This produces
the logarithm in the dependence of the error of the DtN operator on the stability inequality.
Since the Carleman parameter is limited by $h$, the bound of the Fourier transform is obtained
only in a ball limited by $h$, unlike the continuous case.

One of the main difficulties we encounter to deduce the stability estimate is how to develop the key unique continuation estimate for the discrete Calder\'on problem with homogeneous boundary condition.  Due to the "cusps" of the domain under consideration,
we can not use the special weight functions as in \cite{krupchyk2019stability} or \cite{2023Stability} to derive the discrete Carleman estimate with arbitrary boundary data, which is a key tool to prove the unique continuation result (see e.g., \cite{krupchyk2019stability}, \cite{2023Stability} ). Fortunately, inspired by \cite{robbiano1995fonction}, we overcome this difficulty by using a new strategy and deriving  discrete Carleman estimate with other kind of weight functions.  Of course, for the discrete unique continuation estimate again, as expected
in view of the above discussion, the parameters in the  estimate should be limited
in some range depending on the mesh size.

In the study of inverse problems and controllability in the continuous setting, Carleman
estimates have been one of the main tools, and the discrete setting has not been the exception. In recent
years, the development of discrete or semi-discrete Carleman estimates have been used to obtain different results. For instance, the controllability of spatial discrete parabolic systems \cite{2019Insensitizing},  \cite{boyer2010discrete},
\cite{2011Uniform}, \cite{boyer2014carleman}, \cite{ce2022Carleman}, \cite{tnt2014Carleman},
 time discrete parabolic systems \cite{bfh2019Carleman} and the fully discrete case \cite{Casanova2021Carleman}, \cite{lrm2022Discrete}.
 Furthermore,
for inverse problems we find \cite{baudouin2013convergence} for the semi-discrete wave equation in dimension one and  \cite{baudouin2015stability}
for arbitrary dimension.
For the discrete Laplacian operator, we refer to \cite{ervedoza2011uniform} and \cite{lecaros2023discrete}.
Moreover, in the study of discrete propagation of smallness see \cite{Roncal2021Discrete}.
The novelty of the discrete Carleman estimate presented in this paper is that it uses boundary observation and a different weight function. Those discrete or semi-discrete Carleman estimates mentioned above
use internal observations.

In section 2, we give some preliminaries of the discrete problem and present the main
result. In section 3, we  present discrete calculus results and prove a Carleman estimate which is the basis of the proof of the unique continuation result, given in section 4. In section 5, we prove our main result concerning the stability for the discrete Calder\'on problem with
partial data. Finally, in appendix, we focus on the technical steps to obtain the
discrete Carleman estimate from section 3.

\section{Some preliminaries  and main result}
In this section, we introduce notations specific to the discrete problem that will be used throughout this paper.

Let us consider $N\in \mathbb N$, and $h:=\frac{1}{N+1}$ small enough, which represents the size of our mesh. We also define the Cartesian grid of $[0,1]^d$ as $\mathcal K_h:=h[[0,N+1]]^d$, where $[[0,N+1]]^d:=\mathbb Z^d\cap[0,N+1]^d$. We set $\Omega:=(0,1)^d, d\geq3, \Omega_h:=\Omega\cap\mathcal K_h$ and $\partial\Omega_h:=\partial\Omega\cap\mathcal K_h$.

\subsection{Meshes and operators}

For any set of points $\mathcal W_h\subset h\mathbb Z^d$ or $\mathcal W_h\subset \tau_k(h\mathbb Z^d)$, we define the  mesh in the direction $e_k$, with $\{e_k\}_{k=1}^d$ the usual base of $\mathbb R^d$, as $$\mathcal W_{h,k}^*:=\tau_k(\mathcal W_h)\cup\tau_{-k}(\mathcal W_h),$$where $\tau_{\pm k}(\mathcal W_h):=\{x\pm\frac{h}{2}e_k|x\in \mathcal W_h\}$.
Similarly, we define $\mathcal W_{h,kj}^*:=(\mathcal W_{h,k}^*)_j^*=\{x\pm\frac{h}{2}e_k\pm\frac{h}{2}e_j|x\in \mathcal W_h\}$. Obviously, $\mathcal W_{h,kj}^*=\mathcal W_{h,jk}^*$. We will write briefly $\overline{\mathcal W}_{h,k}$ when $k=j$. This enables us to define the boundary points, in the $e_k$ direction, as
$$\partial_k\mathcal W_h:=\overline{\mathcal W}_{h,k}\backslash\mathcal W_h.$$
Moreover, for $1\leq k\leq d$, we define
$$\overline{\mathcal W}_h:=\bigcup_{j=1\atop j\neq k}^d(\tau_j^2(\overline{\mathcal W}_{h,k}))\cup(\tau_{-j}^2(\overline{\mathcal W}_{h,k}))
\quad
{\rm and} \quad\partial\mathcal W_h:=\overline{\mathcal W}_h\backslash \mathcal W_h.$$

For any set of points $\mathcal W_h\subset h\mathbb Z^d$ or $\mathcal W_h\subset \tau_k(h\mathbb Z^d)$,  denote as $\mathcal C(\mathcal W_h)$ the set of discrete functions from $\mathcal W_h $ to $\mathbb R$ and define the difference and average operators in the direction $e_k$ as
$$ (D_ku)(x):=\frac{\tau_ku(x)-\tau_{-k}u(x)}{h}, \quad x\in \mathcal W_{h,k}^*,$$
$$ (A_ku)(x):=\frac{\tau_ku(x)+\tau_{-k}u(x)}{2}, \quad x\in \mathcal W_{h,k}^*,$$
where $\tau_{\pm k}u(x):=u(x\pm \frac{h}{2}e_k)$. We notice that the discrete operators $A_k$ and $D_k$ are defined from $\mathcal C(\overline{\mathcal W}_h)$ to $\mathcal C(\mathcal W_{h,k}^*)$.

Moreover, we define the outward normal of the set $\mathcal W_h$ in the direction $e_k$ as $n_k\in\mathcal C(\partial_k\mathcal W_h)$ by
\begin{align*}
\forall x \in \partial_k\mathcal W_h, n_k(x)=	\begin{cases}
		1 &{\rm if}~ \tau_{-k}(x)\in\mathcal W_{h,k}^*~{\rm and} ~\tau_{k}(x)\notin\mathcal W_{h,k}^* ,\\
		-1 &{\rm if}~ \tau_{-k}(x)\notin\mathcal W_{h,k}^*~{\rm and} ~\tau_{k}(x)\in\mathcal W_{h,k}^* .
	\end{cases}
\end{align*}
Then, we  can define $\partial_k^+\mathcal W_h:=\{x\in\partial_k\mathcal W_h, n_k(x)=1\}$ and $\partial_k^-\mathcal W_h:=\{x\in\partial_k\mathcal W_h, n_k(x)=-1\}$.
Additionally, to introduce the boundary condition we also define the trace operator in the direction $k$, denoted by $t_r^k$, for $u\in\mathcal C(\mathcal W_{h,k}^*)$ as
\begin{align*}
\forall x \in \partial_k\mathcal W_h, t_r^k(u)(x):=	\begin{cases}
		\tau_{-k}u(x) &n_k(x)=1 ,\\
		\tau_{k}u(x) &n_k(x)=-1.
	\end{cases}
\end{align*}

\subsection{The discrete Laplace operator}Similar to \cite{ervedoza2011uniform} and \cite{lecaros2023discrete}, we define the operator $\Delta_h$ from $\mathcal C(\overline\Omega_h)$ to $\mathcal C(\Omega_h)$ as
\begin{align}
	\Delta_hu:=\sum_{k=1}^dD_k(\sigma^kD_ku) ~{\rm in}~\Omega_h,
\end{align}
where $\sigma^k>0$ and $\sigma^k\in \mathcal C((\overline\Omega_{h})_{i}^*),i=1,\dots,d$.  In the sequel, we shall use the symbol $\sigma^k$ for both the continuous function and its sampling on the discrete sets. In fact, from the context, one will be able to deduce the appropriate sampling. In the above expression, with $u $ defined on  $\overline\Omega_h$, it is clear that the function $\sigma^k$ is sampled on  $(\overline\Omega_{h})_{i}^*,i=1,\dots,d,$ as $D_k$ is defined on $(\overline\Omega_{h})_{k}^*$ and the other operator $D_k$ acts on functions defined on $(\overline\Omega_{h})_{k}^*$.

Given a set $\mathcal W_h$, for $u\in\mathcal C(\mathcal W_h)$, we define its $L^\infty_h(\mathcal W_h)$ norm as

$$\|u\|_{L_h^\infty(\mathcal W_h)}:=\max_{x\in \mathcal W_h}\{|u(x)|\}.$$
Then, we introduce the following definition
\begin{definition}
	\begin{align*}
		&\epsilon_d(h):=\sum_{k,j}\|D_j(\sigma^k)\|_{L_h^\infty(\overline\Omega_{h,k})},\quad \epsilon_a(h):=\sum_{k,j}\|A_j(\sigma^k)-1\|_{L_h^\infty(\overline\Omega_{h,k})},
	\end{align*}
	and $
	M(h):=\sum_{k}\|\sigma^k\|_{L_h^\infty((\overline\Omega_{h})_k^*)}$.
\end{definition}
Then we consider the following assumption for the mesh parameters.
\begin{assumption}\label{simg}
	We suppose that
	$$M_0:=\sup_{h\rightarrow0}M(h)<\infty,\quad\epsilon_d,\epsilon_a<1. $$
	\end{assumption}
	We note that if $\epsilon_a<1$, there exist $\underline\sigma$ and $\overline\sigma$, such that,
$$0<\underline\sigma\leq A_j(\sigma^k)\leq\overline\sigma\quad{\rm in}~\overline\Omega_{h,k}.$$
\begin{remark}
	Assumption \ref{simg}  originates from \cite{ervedoza2011uniform} which is used to assure  the existence of CGO solution for the discrete Schr\"odinger equation (see also \cite{lecaros2023discrete}).
\end{remark}

And for any potential $q$ we consider the following:
\begin{assumption}\label{assq}
	The potential $q\in\mathcal C(\Omega_h)$ verifies that the system
	\begin{align*}
		-\Delta_hu+qu&=0,\quad {\rm in}~\Omega_h,\\
		u&=g,\quad {\rm on}~\partial\Omega_h,
	\end{align*}
	has only one solution for any $g\in\mathcal C(\partial\Omega_h)$.
	\end{assumption}

\begin{assumption}\label{as2}
The potential $q\in\mathcal C(\Omega_h)$ verifies that  for the system
\begin{align*}
		-\Delta_hu+qu&=f,\quad {\rm in}~\Omega_h,\\
		u&=0,\quad {\rm on}~\partial\Omega_h,
	\end{align*}
there exists a constant $C>0$ independent of $h$ such that	
	\begin{align*}
		\|u\|_{H_h^1(\Omega_h)}\leq C\|f\|_{L_h^2(\Omega_h)}.
	\end{align*}
\end{assumption}

\begin{remark}
If $q\geq 0$, then we can check that Assumption \ref{as2} holds by using the discrete Poincar\'e inequality.

\end{remark}

Now, let us consider the definition of the discrete normal derivative.
\begin{definition}\label{deno}
	We define the normal derivative in $\partial\Omega_h$, for any $u\in\mathcal C(\overline{\Omega}_h)$,
	$$\partial_nu:=\sum_{k=1}^dt_r^k(\sigma^kD_ku)n_k, \quad {\rm on}~\partial\Omega_h. $$
\end{definition}
Moreover, via the discrete normal derivative, we consider the DtN operator.
\begin{definition}
	We define the DtN map:
	$$\Lambda_h[q](g):=\partial_nu,\quad{\rm on}~\partial\Omega_h, $$
	where $u\in\mathcal C(\overline\Omega_h)$ is the solution of
	\begin{align*}
		-\Delta_hu+qu&=0,\quad {\rm in}~\Omega_h,\\
		u&=g,\quad {\rm on}~\partial\Omega_h.
	\end{align*}
\end{definition}

\subsection{Discrete integral and norm}
For a  finite  set $\mathcal W_h\subset h\mathbb Z^d$ or $\mathcal W_h\subset \tau_k(h\mathbb Z^d)$, we define the discrete integral for $u\in\mathcal C(\mathcal W_h)$ as
$$\int_{\mathcal W_h}u:=h^d\sum_{x\in\mathcal W_h}u(x), $$
and the following $L_h^2$ inner product in $\mathcal C(\mathcal W_h)$
$$(u,v)_{\mathcal W_h}:=\int_{\mathcal W_h}uv, \quad u,v\in\mathcal C(\mathcal W_h).$$

We also define the following norms in $\mathcal C(\mathcal W_h)$:
\begin{align*}
	&\|u\|_{L_h^2(\mathcal W_h)}^2:=(u,u)_{\mathcal W_h},~\quad\quad\quad\quad\quad\quad\quad \|u\|_{\mathring{H}_h^1(\mathcal W_h)}^2:=\sum_{k=1}^d\int_{\mathcal W_{h,k}^*}|D_ku|^2,\\
	 & \|u\|_{H_h^1(\mathcal W_h)}^2:=\|u\|_{L_h^2(\overline{\mathcal W}_h)}^2+\|u\|_{\mathring{H}_h^1(\mathcal W_h)}^2,\quad \|u\|_{\mathring{H}_h^2(\mathcal W_h)}^2:=\sum_{k=1}^d\sum_{j=1}^d\int_{\mathcal W_{h,kj}^*}|D_kD_ju|^2,\\
	&\|u\|_{H_h^2(\mathcal W_h)}^2:=\|u\|_{H_h^1(\mathcal W_h)}^2+\|u\|_{\mathring{H}_h^2(\mathcal W_h)}^2.
\end{align*}

 Next, we introduce the discrete integration on the boundary for a  subset $\Gamma_h$ of $\partial\mathcal W_h$ as
 $$\int_{\Gamma_h} u:=h^{d-1}\sum_{x\in\Gamma_h}u(x),$$
 where $u\in\mathcal C(\Gamma_h)$ and $\Gamma_h\subset\partial\mathcal W_h$. Additionally, we introduce a norm on the boundary:
 $$|f|_{L_h^2(\Gamma_h)}:=\left(h^{d-1}\sum_{x\in\Gamma_h}|f(x)|^2\right)^{1/2},\quad {\rm for~ any}~\Gamma_h\subset\partial\mathcal W_h.$$

 Then, as in \cite{ervedoza2011uniform}, we also define the Sobolev norms for any trace function $g$ in $\mathcal C(\partial\Omega_h)$ as
 \begin{align*}
 	&|g|_{H_h^{1/2}(\partial\Omega_h)}:=\min_{u|_{\partial\Omega_h}=g}\|u\|_{H_h^1(\Omega_h)},\\
 	&|g|_{H_h^{-1/2}(\partial\Omega_h)}:=\max_{|f|_{H_h^{1/2}(\partial\Omega_h)}=1}\int_{\partial\Omega_h}fg.
 \end{align*}

\subsection{Main result}

Before stating our result precisely, let us introduce the discrete $H^r$-norms for $r\in\mathbb R$. For $h>0$, set $\hat{\mathcal  K}_h=[[0,N]]^d:=\mathbb Z^d\cap[0,N]^d$ and define the discrete Fourier transform of a function $u\in\mathcal C(\Omega_h)$, denoted by $\mathcal F(u)$ in $\mathcal C(\hat{\mathcal  K}_h)$ as
$$\mathcal F(u)(\xi):=h^d\sum_{x\in \Omega_h}u(x)e^{-2\pi {\rm i}(x\cdot\xi)}, \quad \forall\xi\in \hat{\mathcal  K}_h.$$
Then, we define the Sobolev-discrete norm $H^r$, for any $r\in\mathbb R$, by
$$|u|_{H_h^r(\Omega_h)}:=\left(\sum_{\xi\in\hat{\mathcal  K}_h}|\mathcal F(u)(\xi)|^2(1+|\xi|^2)^r\right)^{1/2}.$$
In \cite{ervedoza2011uniform}, the authors mentioned that with this definition, the celebrated Plancherel formula states that the norm $|u|_{H^0(\Omega_h)}$ coincides with $\|u\|_{L_h^2(\Omega_h)}$. Moreover, as  in the continuous case, it is classical to show that $\|\cdot\|_{H^1_h(\Omega_h)}$, $\|\cdot\|_{H^2_h(\Omega_h)}$ defined in Section 2.3 and $|\cdot|_{H^1_h(\Omega_h)}$,$|\cdot|_{H^2_h(\Omega_h)}$ are equivalent independently of $h$ respectively.

Now, we present the main result in this paper.
\begin{theorem}\label{theo}
Let  $M>0$. Let $\Gamma$ be a non-empty open subset of $\partial \Omega$  satisfying  $\Gamma_h:=(\Gamma\cap\partial\Omega_h)\subset\partial_i^\pm\Omega_h$, $1\leq i\leq d$. Suppose that Assumption \ref{simg} is satisfied. Let $q_1$, $q_2\in\mathcal C(\Omega_h)$ be two potentials satisfying  Assumption \ref{assq}, Assumption \ref{as2} and  $\|q_j\|_{L_h^\infty(\Omega_h)}\leq M, j=1,2$, $q_1=q_2$ in $\mathcal O_h=\mathcal O\cap\Omega_h$, where $\mathcal O\subset\Omega$ is a neighborhood of $\partial\Omega$ . Set
\begin{align}
	\|\Lambda_h[q_1]-\Lambda_h[q_2]\|_{\mathcal L_h(\Gamma_h)}:=\max_{|g|_{H^{1/2}(\partial\Omega_h)=1}}\left|(\Lambda_h[q_1]-\Lambda_h[q_2])(g)\right|_{H^{-1/2}(\Gamma_h)}.
\end{align}
Then there exists a constant $h_0>0$ depending on $M$, such that for $r>0$ there exists a constant $C>0$ such that
\begin{align}\label{25}
	|q_1-q_2|_{H_h^{-r}(\Omega_h)}\leq C\max\left\{\epsilon_d^{2\alpha}, \epsilon_a^{\alpha}, h^{\alpha}, \left|\ln\left(\|\Lambda_h[q_1]-\Lambda_h[q_2]\|_{\mathcal L_h(\Gamma_h)}\right)\right|^{-2\alpha}\right\},
\end{align}
for all $0<h<h_0$, where $\alpha:=\frac{r}{2r+d}$.
\end{theorem}

\begin{remark}
The stability estimate (\ref{25})  is influenced by $\epsilon_d, \epsilon_a, h$ and the error in the DtN maps.  The term $| \ln(error)|^{-2\alpha}$ on the right-hand side of (\ref{25})  is consistent with the stability estimate for the  continuous case (see e.g., \cite{alessandrini1988stable}). And the terms $\epsilon_d^{2\alpha}, \epsilon_a^{\alpha}, h^{\alpha}$ stem from the discrete setting of the problem.

On the other hand, unlike the stability results for the continuous case, Theorem \ref{theo} does not yield uniqueness. Indeed, if  $\Lambda_h[q_1]=\Lambda_h[q_2]$,  then for any $r>0$, we  have
\begin{align}\label{abc}
	|q_1-q_2|_{H_h^{-r}(\Omega_h)}\leq C\max\left\{\epsilon_d^{2\alpha}, \epsilon_a^{\alpha}, h^{\alpha}\right\}.
\end{align}
Consequently, given $h>0$, this does not
guarantee uniqueness, but rather some kind of asymptotic uniqueness as $h\rightarrow0$, since the right-hand side of (\ref{abc}) goes to zero as $h\rightarrow0$.
\end{remark}

\begin{remark}
In \cite{ervedoza2011uniform}, a discrete stability estimate is obtained by using full boundary data.
	For the partial boundary data case, \cite{lecaros2023discrete} prove a stability result for the Calder\'on problem for a suitable class of
potentials in a cube in $\mathbb R^d, d\geq3$,  when measurements are taken on the boundary  except for one side of the cube. To the best of our knowledge, our result  is the first stability estimate for the discrete Calder\'on problem with
 measurements  taken just on arbitrarily small portions of one side of the cube in $\mathbb R^d, d\geq3$.
\end{remark}

\section{Discrete Carleman estimate}
In this section, we establish a discrete Carleman estimate with boundary observation. Before the calculus, we introduce some formulas which will be used  in the following article.

\subsection{Discrete calculus  formulas}
 First, we provide the product rule for the average and the difference operators.
\begin{lemma}[{\cite[Lemma 2.1]{ervedoza2011uniform}}]\label{le21}
As operators from $\mathcal C(\overline{\mathcal W}_h)$ to $\mathcal C(\mathcal W_{h,kj}^*)$, we have the following identities:
$$A_kA_j=A_jA_k, \quad A_kD_j=D_jA_k, \quad D_kD_j=D_jD_k.$$
Moreover, for any $u,v \in\mathcal C(\overline{\mathcal W}_h)$, we have the following identities on $\mathcal W_{h,k}^*$:
\begin{align*}
D_k(uv)=D_kuA_kv+A_kuD_kv,\\
A_k(uv)=A_kuA_kv+\frac{h^2}{4}D_kuD_kv.	
\end{align*}
\end{lemma}

As a direct consequence of Lemma \ref{le21}, we have the following Corollary.
\begin{corollary}\label{co22}
	For $u\in\mathcal C(\overline{\mathcal W}_h)$,
	$$A_k(u^2)=(A_ku)^2+\frac{h^2}{4}(D_ku)^2, \quad D_k(u^2)=2D_kuA_ku, \quad {\rm on}~\mathcal W_{h,k}^*.$$
	In particular, for all $u\in\mathcal C(\overline{\mathcal W}_h)$,
	$$(A_ku)^2\leq A_k(u^2), \quad (D_ku)^2\leq \frac{4}{h^2}A_k(u^2), \quad {\rm on}~\mathcal W_{h,k}^*.$$
	\end{corollary}

 Then we state  the discrete integration by parts for the difference and average operators.
 \begin{lemma}[{\cite[Lemma 2.2]{lecaros2023discrete}}]\label{le28}
 	Given a  mesh $\mathcal W_h $ defined as above, for any $v\in\mathcal C(\mathcal W_{h,k}^*)$, $u\in\mathcal C(\overline{\mathcal W}_{h,k})$, we have
 	\begin{align}
 		&\int_{\mathcal W_h}uD_kv=-\int_{\mathcal W_{h,k}^*}vD_ku+\int_{\partial_k\mathcal W_h}ut_r^k(v)n_k,\label{dk}\\
 		&\int_{\mathcal W_h}uA_kv=\int_{\mathcal W_{h,k}^*}vA_ku-\frac{h}{2}\int_{\partial_k\mathcal W_h}ut_r^k(v)\label{ak}.
 	\end{align}
 \end{lemma}

Finally we give the discrete version of Green's identity similarly to \cite{ervedoza2011uniform}.
\begin{proposition}\label{le29}
	 For all $u,v\in\mathcal C(\overline\Omega_h)$, we have
	\begin{align}\label{e23}
		\int_{\Omega_h}(\Delta_hu)v=-\sum_{k=1}^d\int_{\Omega_{h,k}^*}\sigma^kD_kuD_kv+\int_{\partial\Omega_h}(\partial_nu)v.
	\end{align}

\end{proposition}
\begin{proof}
Choose a set $\overline Q_h$ satisfying $\overline\Omega_h\subset Q_h\subset\overline Q_h\subset h\mathbb Z^d$,
	  then $\Omega_{h,k}^*\subset Q_{h,k}^*$ so that we can define an operator $I_h$ from $\mathcal C(\Omega_{h,k}^*)$ to $\mathcal C(Q_{h,k}^*)$ that extends functions by zero outside $\Omega_{h,k}^*$. Extending $n_k$ by 0 outside $\partial_k\Omega$, for all $u\in\mathcal C(\overline\Omega_h)$,
	  we have
	\begin{align}\label{e22}
		-h\sum_{k=1}^dD_k(I_h(\sigma^kD_ku)) =\sum_{k=1}^dt_r^k(\sigma^kD_ku)n_k=\partial_nu\quad {\rm on}~\partial\Omega_h.
	\end{align}
	  The function $D_k(I_h(\sigma^kD_ku))$ belongs to  $\mathcal C(\overline Q_{h,k})$. Due to the inclusions
	$$\partial_k\Omega_h\subset\overline\Omega_{h,k}\subset\overline Q_{h,k},$$
	it makes sense to look at the values of $D_k(I_h(\sigma^kD_ku))$ on $\partial_k\Omega_h$. Set $x\in\partial_k\Omega_h$, then either $n_k(x)=-1$ or $1$. Since the two cases are treated the same way, we may suppose that $n_k(x)=1$, that is $x-\frac{h}{2}e_k\in\Omega_{h,k}^*$ and
	$x+\frac{h}{2}e_k\notin\Omega_{h,k}^*$. In this case, $(I_h(\sigma^kD_ku))(x+\frac{h}{2}e_k)=0$ and
	$$D_k(I_h(\sigma^kD_ku))=-\frac{1}{h}(\sigma^kD_ku)(x-\frac{h}{2}e_k)=-\frac{1}{h}t_r^k(\sigma^kD_ku),$$
	then we get (\ref{e22}).
	
	We now turn our attention to proving Green's formula. Denote by $\widetilde I_h$ the operator from $\mathcal C(\overline\Omega_h)$ to $\mathcal C(\overline Q_h)$ that extend functions by zero outside $\overline\Omega_h$. We have
	$$D_k(\widetilde I_hv)=D_kv\quad {\rm on} ~\Omega_{h,k}^*, \quad{\rm and}\quad I_h(\sigma^kD_ku)=0\quad {\rm outside}~\Omega_{h,k}^*.$$
	Hence
	$$\sum_{k=1}^d\int_{\Omega_{h,k}^*}\sigma^kD_kuD_kv=\sum_{k=1}^d\int_{Q_{h,k}^*}I_h(\sigma^kD_ku)D_k(\widetilde I_hv).$$
	Since $\overline\Omega_h\in Q_h$, $\widetilde I_hv=0$ on $\partial Q_h$, and then, performing the integration by parts (\ref{dk}),
	\begin{align*}
		\sum_{k=1}^d\int_{\Omega_{h,k}^*}\sigma^kD_kuD_kv&=-\sum_{k=1}^d\int_{Q_{h}}\widetilde I_hvD_k(I_h(\sigma^kD_ku))\\
		&=-\sum_{k=1}^d\int_{\overline\Omega_h}vD_k(I_h(\sigma^kD_ku))\\
		&=-\sum_{k=1}^d\int_{\Omega_h}vD_k(\sigma^kD_ku)-\sum_{k=1}^dh\int_{\partial\Omega_h}vD_k(I_h(\sigma^kD_ku))\\
		&=-\int_{\Omega_h}v\Delta_hu+\int_{\partial\Omega_h}v\partial_nu.
	\end{align*}	
In the last equality, we have used (\ref{e22}). This completes the proof.

\end{proof}

\subsection{Calculus results related to the weight functions}

  Our Carleman weight function is defined as $e^{s\varphi}$ for $s\geq 1$, with $\varphi=e^{-\lambda\psi}$,   $\lambda\geq1$, where $\psi$ is a continuous function. In this paper, we shall use the same  notation for the sample of the continuous function on the discrete sets.
We suppose that the function $\psi$ satisfies the following properties.
\begin{assumption}\label{as31}
Let $\widetilde\Omega$ be an open and connected bounded neighborhood of $\overline\Omega$ in $\mathbb R^d$ with smooth boundary. The function $\psi$ is  in $C^\infty(\overline{\widetilde\Omega})$, i.e. infinitely differentiable function,   and satisfies
$$|\nabla\psi|>0~ {\rm and}~
	\psi>0 ~{\rm in}~ \widetilde\Omega.$$
\end{assumption}
\begin{remark}
	We enlarge the open set $\Omega$ to a larger open set $\widetilde\Omega$ as this will guarantee that after performing some discrete operations in $\Omega$, $\psi$ is  still well-defined.
\end{remark}

In the sequel, $C$ will denote a generic constant independent of $h$, whose value may change from line to line. As usual, we shall denote by $O(1)$ a bounded function and by $O_\lambda(1)$ a bounded function  once $\lambda$ is fixed. We denoted by $O_\lambda(sh)$ the functions that verify $\|O_\lambda(sh)\|_{L_h^\infty(\Omega_h)}\leq C_\lambda sh$ for some constant $C_\lambda$ depending on $\lambda$.

We say that $\alpha$ is a multi-index if $\alpha=(\alpha_1,\dots,\alpha_d)\in \mathbb N^d$ and for $y=(y_1,\dots,y_d)\in\mathbb R^d$ we write:
$$|\alpha|=\alpha_1+\cdots+\alpha_d,\quad y^\alpha=y_1^{\alpha_1}\cdots y_d^{\alpha_d}.$$
Moreover, we denote by $\partial$ the differential operator in the continuous case. For $f=f(x_1, \dots,x_d)\in C^{|\alpha|}(\mathbb R^d)$, we write
$$\quad\partial^\alpha f=\partial^{\alpha_1}_1\cdots\partial^{\alpha_d}_df, $$
where we abbreviate $\partial_{x_i}^{\alpha_i}$ as $\partial_i^{\alpha_i}$, $i=1,\dots,d$.

We now provide some technical lemmas related to discrete operations performed on the Carleman weight functions that is of the form $e^{s\varphi}$ with $\varphi=e^{-\lambda\psi}$, $\psi\in C^\infty$. For concision, we set $r=e^{s\varphi}$ and $\rho=r^{-1}$. The positive parameters $s$ and $h$ will be large and small respectively and we are particularly interested in the dependence on $s,h$ and $\lambda$ in the following basic estimates.

We assume $s\geq 1$ and $\lambda\geq1$.  We shall use multi-index of the form $\alpha=(\alpha_1,\dots, \alpha_d)$. The proofs can be found in \cite{boyer2010discrete1d}, \cite{boyer2010discrete} and \cite{boyer2014carleman}.
\begin{lemma}[{\cite[Lemma 2.7]{boyer2010discrete}}]
\label{le32}
	Let $\alpha$ and $\beta$ be multi-indices. We have
	\begin{align*}
		\partial^\beta(r\partial^\alpha\rho)=&|\alpha|^{|\beta|}(-s\varphi)^{|\alpha|}\lambda^{|\alpha+\beta|}(-\nabla\psi)^{\alpha+\beta}\\
		&+|\alpha||\beta|(s\varphi)^{|\alpha|}\lambda^{|\alpha+\beta|-1} O(1)+s^{|\alpha|-1}|\alpha|(|\alpha|-1)O_\lambda(1)\\
		=&O_\lambda(s^{|\alpha|}).
	\end{align*}
	The same expressions hold with $r$ and $\rho$ interchanged and with $s$ changed into $-s$.
\end{lemma}

\begin{lemma}[{\cite[Corollary 2.8]{boyer2010discrete}}]\label{le33}
	Let $\alpha, \beta$ and $\delta$ be multi-indices. We have
	\begin{align*}
		\partial^\delta(r^2(\partial^\alpha\rho)\partial^\beta\rho)=&|\alpha+\beta|^{|\delta|}(-s\varphi)^{|\alpha+\beta|}\lambda^{|\alpha+\beta+\delta|}(-\nabla\psi)^{\alpha+\beta+\delta}\\
		&+|\delta||\alpha+\beta|(s\varphi)^{|\alpha+\beta|}\lambda^{|\alpha+\beta+\delta|-1}O(1)\\&+s^{|\alpha+\beta|-1}(|\alpha|(|\alpha|-1)+|\beta|(|\beta|-1))O_\lambda(1)\\
		=&O_\lambda(s^{|\alpha+\beta|}).
	\end{align*}
\end{lemma}
\begin{lemma}[{\cite[Proposition 2.9]{boyer2010discrete}}]\label{le34}
	Let $\alpha$ be a multi-index. Let $k,j=1,\dots,d$, provided $sh\leq \mathcal R$, we have
	\begin{align*}
rA_k^p\rho&=1+O_{\lambda, \mathcal R}((sh)^2)=O_{\lambda, \mathcal R}(1), \quad p=1,2,\\
		rA_k^pD_k\rho &=r\partial_k\rho+sO_{\lambda,\mathcal R}((sh)^2)=sO_{\lambda,\mathcal R}(1),\quad p=0,1,\\
		rD_k^{p_k}D_j^{p_j}\rho&=r\partial_k^{p_k}\partial_j^{p_j}\rho+s^2O_{\lambda,\mathcal R}((sh)^2)=s^2O_{\lambda,\mathcal R}(1), \quad p_k+p_j\leq2.
	\end{align*}
	The same estimates hold with $\rho$ and $r$ interchanged.
\end{lemma}
\begin{lemma}[{\cite[Proposition 2.12]{boyer2010discrete}}]\label{lexx}
	Let $p\in\mathbb N$, $k,j=1,\dots,d$, provided $sh\leq \mathcal R$, we have
	\begin{align*}
	D_k^{p}(rD_k^2\rho)&=s^2O_{\lambda,\mathcal R}(1),\\
		D_k^{p}(rA_k^2\rho)&=O_{\lambda,\mathcal R}((sh)^2).
	\end{align*}
	The same estimates hold with $r$ and $\rho$ interchanged.
\end{lemma}
\begin{lemma}[{\cite[Proposition 2.13]{boyer2010discrete}}]\label{le35}
	Let $\alpha$ be a multi-index, $k,j=1,\dots,d$ and $p_k, p_k',p_j,p_j'\in \mathbb N$. For $p_k+p_j\leq2$, provided $sh\leq\mathcal R$ we have
	\begin{align*}
		&A_k^{p_k'}A_j^{p_j'}D_k^{p_k}D_j^{p_j}(r^2A_kD_k\rho D_j^2\rho)=\partial_k^{p_k}\partial_j^{p_j}(r^2(\partial_k\rho)\partial_j^2\rho)+s^3O_{\lambda,\mathcal R}((sh)^2)=s^3O_{\lambda,\mathcal R}(1),\\
		&A_k^{p_k'}A_j^{p_j'}D_k^{p_k}D_j^{p_j}(r^2A_jD_j\rho A_k^2\rho)=\partial_k^{p_k}\partial_j^{p_j}(r\partial_j\rho)+sO_{\lambda,\mathcal R}((sh)^2)=sO_{\lambda,\mathcal R}(1).
	\end{align*}
\end{lemma}

\subsection{Discrete Carleman estimate}

We introduce the following notations related to the coefficients $\sigma^k,k=1,\dots,d$ for any function $f$,
$$\nabla_\sigma f=(\sqrt{\sigma^1}\partial_1f,\dots,\sqrt{\sigma^d}\partial_df)^T, \quad \Delta_\sigma f=\sum_{k=1}^d\sigma^k\partial^2_kf.$$

The Carleman weight function is of the form $r=e^{s\varphi}$ with $\varphi=e^{-\lambda\psi}$, where $\psi$ satisfies Assumption \ref{as31}. We recall that $\rho=r^{-1}$.

The enlarged neighborhood $\widetilde\Omega$ of $\Omega$ introduced in Assumption \ref{as31} allows us to apply multiple discrete operators such as $D_k$ and $A_k$ on the weight functions.

We are now in position to state and prove the following discrete Carleman estimate.

\begin{proposition}[Discrete Carleman estimate]\label{Carleman}Suppose that Assumption \ref{simg} is satisfied.
Let $\psi$ be a function verifying Assumption \ref{as31} and let $q\in \mathcal C(\Omega_h)$ satisfying $\|q\|_{L_h^\infty(\Omega_h)}\leq M$.  Then there exists a constant  $\lambda_0\geq 1$ such that for each $\lambda\geq \lambda_0$, there exist $s_0(\lambda)\geq 1$, $0<h_0<1$, $\varepsilon_0>0$  and $C = C(\varepsilon_0,s_0,\lambda,M) $ independent of $h > 0$ such that
\begin{align}\label{ce}\begin{split}
	s^3\|e^{s\varphi}u\|_{L^2_h(\Omega_h)}^2&+s\sum_{k=1}^d\int_{\Omega_{h,k}^*}e^{2s\varphi}|D_ku|^2\\
	&\leq C\left(\|e^{s\varphi}(-\Delta_h+q)u\|_{L^2_h(\Omega_h)}^2+s\sum_{k=1}^d\int_{\partial_k\Omega_h}t_r^k(e^{2s\varphi})|\partial_nu|^2\right),
	\end{split}
\end{align}
for all $s\geq s_0(\lambda)$, $0< h\leq h_0$, $sh\leq \varepsilon_0$ and $u\in\mathcal C(\overline\Omega)$ satisfying $u|_{\partial\Omega_h}=0$.
\end{proposition}
\begin{remark}
	The parameter $s$ is limited by the condition $sh\leq\varepsilon_0$. This restriction on the range of the Carleman parameter appears in discrete Carleman estimates (see e.g., \cite{baudouin2013convergence},  \cite{baudouin2015stability},
	\cite{boyer2010discrete1d}, \cite{boyer2010discrete}, \cite{boyer2014carleman}, \cite{ervedoza2011uniform}). This is related to the fact that the conjugation of discrete operators with the exponential weight behaves as in the continuous case only for $sh$ small enough, since for instance
	$$e^{s\varphi}D_k(e^{-s\varphi})\simeq -s\partial_k\varphi\quad {\rm only~for}~sh~{\rm small~ enough}.$$

\end{remark}
\begin{remark}
Several recent works have been concerned with discrete and semi-discrete Carleman estimates for second-order differential operators  (see e.g.,  \cite{baudouin2013convergence},  \cite{baudouin2015stability}, \cite{boyer2010discrete1d}, \cite{boyer2010discrete}, \cite{boyer2014carleman}, \cite{ervedoza2011uniform}, \cite{lecaros2021stability},  \cite{lecaros2023discrete}). The main differences between our estimate and the existing results are  that  we choose a different weight function and a more general discretization for the Laplacian. At the same time, we consider boundary observation.
\end{remark}

\begin{proof}[Proof of Proposition \ref{Carleman}]
We make the change of variable $v=ru$. Our first task is to obtain an expression for $ \Delta_{h,\varphi} v:=r\Delta_h(\rho v)$ with the change of variable proposed. By using Lemma \ref{le21}, we have
	\begin{align}\label{e1}
	\begin{split}
		\Delta_{h,\varphi} v=\sum_{k=1}^dr(&D_k(\sigma^kD_k\rho)A_k^2v+A_k(	\sigma^kD_k\rho )D_kA_kv\\&+D_kA_k\rho A_k(\sigma^kD_kv)+A_k^2\rho D_k(\sigma^kD_kv)),
		\end{split}
	\end{align}
	and
	\begin{align*}
	&A_k(\sigma^kD_kv)=A_k\sigma^kA_kD_kv+\frac{h}{4}	D_k\sigma^k(\tau_kD_kv-\tau_{-k}D_kv),\\
	&A_k(\sigma^kD_k\rho)=A_k\sigma^kA_kD_k\rho+\frac{h^2}{4}D_k\sigma^kD_k^2\rho,\\
	&	D_k(\sigma^kD_k\rho)=D_k\sigma^kA_kD_k\rho+A_k\sigma^kD_k^2\rho.
	\end{align*}
	Recalling that, in each term, $\sigma^k$ is the sampling of the given continuous coefficient $\sigma^k$ on the discrete sets, so that we can deduce from Assumption \ref{simg} that $\|A_k\sigma^k-\sigma^k\|_\infty\leq Ch$. Then,
we set 		
\begin{align}\label{abg}
\begin{split}
	A_1v=&\sum_{k=1}^drA_k^2\rho D_k(\sigma^kD_kv),\quad\quad\quad
	A_2v=\sum_{k=1}^d\sigma^krD_k^2\rho A_k^2v,\\
	B_1v=&2\sum_{k=1}^d\sigma^krA_kD_k\rho A_kD_kv,\quad\quad
	B_2v=-2s(\Delta_\sigma\varphi) v
	\end{split}
\end{align}
and
\begin{align*}
\begin{split}
	g=&\Delta_{h,\varphi} v-\sum_{k=1}^d\frac{h}{4}rD_kA_k\rho D_k\sigma^k (\tau_kD_kv-\tau_{-k}D_kv)\\
	&-\sum_{k=1}^d\frac{h^2}{4}rD_k^2\rho D_k\sigma^k D_kA_kv-h\sum_{k=1}^dO(1)rA_kD_k\rho A_kD_kv\\
	&-\sum_{k=1}^d\left(rD_kA_k\rho D_k\sigma^k+hO(1)rD_k^2\rho \right)A_k^2v-2s(\Delta_\sigma\varphi) v.
	\end{split}
\end{align*}
	As in \cite{fursikov1996controllability}, we set
$$Av=A_1v+A_2v, \quad Bv=B_1v+B_2v.$$
 An explanation for the introduction of this additional term $B_2v$ is provided in \cite{le2009carleman}.
Thus, equation (\ref{e1})  reads $Av+Bv=g$ and we have
\begin{align}\label{e33}
	\|Av\|^2_{L^2_h(\Omega_h)}+\|Bv\|^2_{L^2_h(\Omega_h)}+2(Av,Bv)_{\Omega_h}=\|g\|^2_{L^2_h(\Omega_h)}.
\end{align}
We shall need the following estimation of 	$\|g\|^2_{L^2_h(\Omega_h)}$. The proof can be found in Appendix A.
	
\begin{lemma}\label{lerhs}
For $sh\leq \mathcal R$, we have
\begin{align}
	\|g\|^2_{L^2_h(\Omega_h)}\leq C_{\lambda, \mathcal R}\left(\|\Delta_{h,\varphi} v\|^2_{L^2_h(\Omega_h)}+s^2\|v\|^2_{L^2_h(\Omega_h)}+(sh)^2\|v\|_{\mathring{H}_h^1(\Omega_h)}^2\right).
\end{align}
	
\end{lemma}

Now, we shall estimate the inner-product
\begin{align}\label{e36}
	(Av,Bv)_{\Omega_h}=\sum_{i,j=1}^2(A_iv,B_jv)_{\Omega_h}.
\end{align}
To estimate each term of (\ref{e36}), we obtain the following results Lemma \ref{le37}-Lemma \ref{le311}. To make the paper more concise, we put the proofs of Lemma \ref{le37}-Lemma \ref{le311} to Appendix A.

\begin{lemma}\label{le37}
For $sh\leq \mathcal R$, we have
	\begin{align*}
	(A_1v,B_1v)_{\Omega_h}\geq &-\sum_{k=1}^d\int_{\Omega_{h,k}^*}s\lambda^2\varphi\sigma^k|\nabla_\sigma\psi|^2 |D_kv|^2+Y_1+Z_1,
	\end{align*}
where
\begin{align*}
	Y_1
	=&\sum_{k=1}^d\int_{\partial_k\Omega_h}\gamma_1^k\beta_1^k\sigma^kA_k\sigma^kt_r^k(|D_kv|^2)n_k-h\sum_{k=1}^d\int_{\partial_k\Omega_h}t_r^k(D_k(\gamma_1^k\beta_1^k\sigma^kA_k\sigma^k)|D_kv|^2)
	\\& -\frac{h^2}{2}\sum_{k=1}^d\int_{\partial_k\Omega_h}D_k^2(\gamma_1^k\beta_1^k\sigma^kA_k\sigma^k)t_r^k(|D_kv|^2)n_k,
\end{align*}
with $\gamma_1^k=rA_k^2\rho$, $\beta_1^k=rA_kD_k\rho$ and
\begin{align*}
	Z_1=&-\sum_{k=1}^d\int_{\Omega_{h,k}^*}(s\lambda\varphi |O(1)|+s|O_{\lambda\mathcal, R}(sh)| )|D_kv|^2\\
	&-\sum_{k=1}^d\int_{\Omega_h}h^2(s\lambda\varphi |O(1)|+s|O_{\lambda\mathcal, R}(sh)|)|D_k^2v|^2\\
	&-\sum_{k=1}^d\int_{\Omega_h}(s\lambda\varphi |O(1)|+s|O_{\lambda\mathcal, R}(sh)| )|A_kD_kv|^2\\
	&-\sum_{j,k=1 \atop k\neq j}^d\int_{\Omega_{h,kj}^*}h^2(s\lambda\varphi |O(1)|+s|O_{\lambda,\mathcal R}(sh)|)|D_kD_jv|^2.
\end{align*}
\end{lemma}

\begin{lemma}\label{le39}
For $sh\leq \mathcal R$, we have
$$(A_2v,B_1v)_{\Omega_h}\geq 3\int_{\Omega_h} s^3\lambda^4\varphi^3|\nabla_\sigma\psi|^4|v|^2+Y_2+Z_2,$$
where
$$Y_2=\sum_{k=1}^d\int_{\partial_k\Omega_h}\gamma_2^k \beta_1^k(\sigma^k)^2t_r^k(|A_kv|^2)n_k,$$
with $\gamma_2^k=rD_k^2\rho$, $\beta_1^k=rA_kD_k\rho$ and
\begin{align*}
	Z_2=&-\int_{\Omega_h}\left((s\lambda\varphi)^3|O(1)|+s^2|O_{\lambda,\mathcal R}(1)|+s^3|O_{\lambda,\mathcal R}(sh)|\right)|v|^2\\
	&-\sum_{k=1}^d\int_{\Omega_h}s|O_{\lambda,\mathcal R}(sh)||D_kA_kv|^2-\sum_{k=1}^d\int_{\Omega_{h,k}^*}s|O_{\lambda,\mathcal R}((sh)^2)||D_kv|^2.
\end{align*}
	
\end{lemma}

\begin{lemma}\label{le310}
	For $sh\leq \mathcal R$, we have
\begin{align*}
	(A_1v,B_2v)_{\Omega_h}\geq 2\sum_{k=1}^d\int_{\Omega_{h,k}^*}s\lambda^2\varphi\sigma^k|\nabla_\sigma\psi|^2|D_kv|^2+Z_3,
\end{align*}
where $$Z_3=-\int_{\Omega_h}s^2|O_{\lambda,\mathcal R}(1)||v|^2-\sum_{k=1}^d\int_{\Omega_{h,k}^*}(s\lambda\varphi |O(1)|+s|O_{\lambda,\mathcal R}(sh)|)|D_kv|^2.$$
\end{lemma}	
\begin{lemma}\label{le311}
	For $sh\leq \mathcal R$, we have
	\begin{align*}
	(A_2v,B_2v)_{\Omega_h}\geq -2\int_{\Omega_h} s^3\lambda^4\varphi^3|\nabla_\sigma\psi|^4|v|^2+Z_4,
\end{align*}
where
\begin{align*}
	Z_4=-\int_{\Omega_h}\left((s\lambda\varphi)^3|O(1)|+s^2|O_\lambda(1)|+s^3|O_{\lambda,\mathcal R}((sh)^2)|\right)|v|^2-\sum_{k=1}^d\int_{\Omega_{h,k}^*}s|O_{\lambda,\mathcal R}(sh)||D_kv|^2.
\end{align*}
\end{lemma}

 Combining the estimates in Lemma \ref{lerhs}-Lemma \ref{le311} with (\ref{e33}), for $sh\leq \mathcal R$, we have
\begin{align}\label{e37}
\begin{split}
	&\int_{\Omega_h} 2s^3\lambda^4\varphi^3|\nabla_\sigma\psi|^4|v|^2+\sum_{k=1}^d\int_{\Omega_{h,k}^*}2s\lambda^2\varphi\sigma^k|\nabla_\sigma\psi|^2|D_kv|^2+2(Y_1+Y_2+Z)\\
	\leq &C_{\lambda, \mathcal R}\left(\|\Delta_{h,\varphi} v\|^2_{L^2_h(\Omega_h)}+s^2\|v\|^2_{L^2_h(\Omega_h)}+(sh)^2\|v\|_{\mathring{H}_h^1(\Omega_h)}^2\right),
	\end{split}
\end{align}
where $Z=Z_1+Z_2+Z_3+Z_4$.
\begin{lemma}\label{le313}
	For $sh\leq \mathcal R$, we have
	\begin{align*}
\sum_{k=1}^d\int_{\Omega_{h,k}^*}s\lambda^2\varphi\sigma^k|\nabla_\sigma\psi|^2|D_kv|^2
	\geq& \frac{h^2}{4d}\sum_{j,k=1 \atop j\neq k}^d\int_{\Omega_{h,kj}^*}s\lambda^2\varphi\sigma^k|\nabla_\sigma\psi|^2|D_jD_kv|^2\\
	&+\frac{1}{d}\sum_{k=1}^d\int_{\Omega_h}s\lambda^2\varphi\sigma^k|\nabla_\sigma\psi|^2|A_kD_kv|^2\\
	&+\frac{h^2}{4d}\sum_{k=1}^d\int_{\Omega_h}s\lambda^2\varphi\sigma^k|\nabla_\sigma\psi|^2|D_k^2v|^2+Y_3+Z_5,
\end{align*}	
where $$Y_3=	\frac{h^2}{4}\sum_{k=1}^d\int_{\partial_k\Omega_h}D_k(\varphi\sigma^k|\nabla_\sigma\psi|^2)t_r^k(|D_kv|^2)n_k,$$	
and
\begin{align*}
	Z_5=&-\sum_{j,k=1 \atop j\neq k}^d\int_{\Omega_{h,kj}^*}h^2|O_{\lambda,\mathcal R}(sh)| |D_jD_kv|^2-\sum_{k=1}^d\int_{\Omega_h}|O_{\lambda,\mathcal R}(sh)||A_kD_kv|^2\\
	&-\sum_{k=1}^d\int_{\Omega_h}h^2|O_{\lambda,\mathcal R}(sh)||D_k^2v|^2-\sum_{k=1}^d\int_{\Omega_{h,k}^*}|O_{\lambda,\mathcal R}(sh)||D_kv|^2.
\end{align*}	
\end{lemma}
The proof of Lemma \ref{le313} can be found in Appendix A.

Combining (\ref{e37}) with Lemma \ref{le313},  there exists a constant $\lambda_0\geq1$ sufficiently large, for $\lambda\geq\lambda_0$, there exist $s_0(\lambda)\geq 1$ sufficiently large, $\varepsilon_0,0<h_0<1$ sufficiently small and  $\widehat{C}_{\lambda,\varepsilon_0,s_0}, \widetilde{C}_{\lambda,\varepsilon_0,s_0}>0$,  such that for  $s\geq s_0(\lambda)$, $0<h\leq h_0$  and $sh\leq \varepsilon_0$, we have
\begin{align}\label{e7}
\begin{split}
	\widehat{C}_{\lambda,\varepsilon_0,s_0}\left(s^3\|v\|^2_{L^2_h(\Omega_h)}+s\|v\|_{\mathring{H}_h^1(\Omega_h)}^2\right)
		\leq \widetilde{C}_{\lambda,\varepsilon_0,s_0}\|\Delta_{h,\varphi} v\|^2_{L^2_h(\Omega_h)}-2(Y_1+Y_2)-Y_3.
	\end{split}
\end{align}
Now we  estimate the boundary terms.  Since $v=0$ on $\partial_k\Omega_h$, we have $t_r^k(A_kv)=-\frac{h}{2}t_r^k(D_kv)n_k$ for all $x\in \partial_k\Omega_h$, which leads to $t_r^k(|A_kv|^2)=\frac{h^2}{4}t_r^k(|D_kv|^2)$ by virtue of $t_r^k(|A_kv|^2)=|t_r^k(A_kv)|^2$. Thus, we have
\begin{align*}
	-2(Y_1+Y_2)-Y_3=\sum_{k=1}^d\int_{\partial_k\Omega_h}&\left(-2\gamma_1^k\beta_1^k\sigma^kA_k\sigma^kt_r^k(|D_kv|^2)n_k+ht_r^k(D_k(\gamma_1^k\beta_1^k\sigma^kA_k\sigma^k))t_r^k(|D_kv|^2)\right.\\
	& \left.-\frac{h^2}{4}D_k(\varphi\sigma^k|\nabla_\sigma\psi|^2)t_r^k(|D_kv|^2)n_k-\frac{h^2}{2}\gamma_2^k \beta_1^k(\sigma^k)^2t_r^k(|D_kv|^2)n_k\right.\\
	&\left.+h^2D_k^2(\gamma_1^k\beta_1^k\sigma^kA_k\sigma^k)t_r^k(|D_kv|^2)n_k\right),
\end{align*}
where $\gamma_1^k=rA_k^2\rho$,$\gamma_2^k=rD_k^2\rho$, $\beta_1^k=rA_kD_k\rho$.
Similar to Lemma \ref{leA1}, we have following results.
\begin{align*}
&\gamma_1^k\beta_1^k\sigma^kA_k\sigma^k=s\lambda\varphi(\sigma^k)^2\partial_k\psi+sO_{\lambda,\varepsilon_0}(sh),\quad\\
&D_k^2(\gamma_1^k\beta_1^k\sigma^kA_k\sigma^k)=sO_{\lambda,\varepsilon_0}(1),\quad t_r^k(D_k(\gamma_1^k\beta_1^k\sigma^kA_k\sigma^k))=sO_{\lambda,\varepsilon_0}(1),\\
&D_k(\varphi\sigma^k|\nabla_\sigma\psi|^2)=O_{\lambda}(1),\quad \quad\quad\gamma_2^k \beta_1^k(\sigma^k)^2=s^3O_{\lambda,\varepsilon_0}(1).
\end{align*}
Thus, we have
\begin{align}\label{boun}
	\begin{split}
		-2(Y_1+Y_2)-Y_3\leq\sum_{k=1}^d\int_{\partial_k\Omega_h}&\left(2s\lambda\varphi(\sigma^k)^2|\partial_k\psi| +s|O_{\lambda,\varepsilon_0}(sh)|\right.\\
	& \left.+h^2|O_{\lambda}(1)|
	+h|O_{\lambda,\varepsilon_0}(sh)|\right)t_r^k(|D_kv|^2).
	\end{split}
\end{align}
Combining (\ref{e7}) with (\ref{boun}), there exists constant $C>0$ such that
\begin{align}\label{e39}
	s^3\|v\|^2_{L^2_h(\Omega_h)}+s\|v\|_{\mathring{H}_h^1(\Omega_h)}^2
		\leq C\left(\|\Delta_{h,\varphi} v\|^2_{L^2_h(\Omega_h)}+s\sum_{k=1}^d\int_{\partial_k\Omega_h}t_r^k(|D_kv|^2)\right),
\end{align}
for  $s\geq s_0(\lambda)$, $0<h\leq h_0$      and $sh\leq \varepsilon_0$,
Finally, we need to express all the terms in the estimate above in terms of the original function $u$.  To this end, we need the following Lemma.
\begin{lemma}\label{le314}
	For $sh\leq \mathcal R$, we have
	\begin{align}
		\sum_{k=1}^d\int_{\Omega_{h,k}^*}r^2|D_ku|^2&\leq C_{\lambda,\mathcal R}\left(s^2\|v\|^2_{L^2_h(\Omega_h)}+\|v\|_{\mathring{H}_h^1(\Omega_h)}^2\right),\label{e310}\\
		\sum_{k=1}^d\int_{\partial_k\Omega_h}t_r^k(|D_kv|^2)&\leq  C_{\lambda,\mathcal R}\sum_{k=1}^d\int_{\partial_k\Omega_h}t_r^k(r^2)|\partial_nu|^2. \label{e311}
	\end{align}
\end{lemma}
The proof of Lemma \ref{le314} can be found in Appendix A.

Recalling that
$q\in \mathcal C(\Omega_h)$ satisfying $\|q\|_{L_h^\infty(\Omega_h)}\leq M$
and combining (\ref{e39}) with Lemma \ref{le314},  we obtain
\begin{align*}
	s^3\|ru\|^2_{L^2_h(\Omega_h)}+s\sum_{k=1}^d&\int_{\Omega_{h,k}^*}r^2|D_ku|^2\\
	\leq &C_{
\lambda,\varepsilon_0,s_0, M}\left(\|r(-\Delta_{h}u +qu)\|^2_{L^2_h(\Omega_h)}+s\sum_{k=1}^d\int_{\partial_k\Omega_h}t_r^k(r^2)|\partial_nu|^2\right),
\end{align*}
for
 $s\geq s_0(\lambda)$, $0<h\leq h_0$      and $sh\leq \varepsilon_0$.

Thus the proof of Proposition \ref{Carleman} is complete.
\end{proof}

\section{A unique continuation estimate}
In this section, we prove a unique continuation result for the discrete Schr\"odinger equation with homogeneous boundary condition which will be a crucial step in the proof of the main result. Our argument inspired by the idea  of \cite{robbiano1995fonction}, see also \cite{bellassoued2004global} and \cite{bellassoued2006logarithmic}.
By applying the discrete Carleman estimate with full boundary observation, we show a unique continuation result  from  boundary data on an arbitrary small part $\Gamma_h\subset \partial_i^\pm\Omega_h $, i=1,\dots,d.

Before presenting the  result, we introduce some notations. Let $\mathcal O\subset\Omega$ be an
arbitrary neighborhood of $\partial\Omega$. We choose $\varrho, \varrho_1,\varrho_2>0$ such that
\begin{align}
	\omega(8\varrho):=\{x\in \Omega, {\rm dist}(x,\partial\Omega)<8\varrho\}\subset\mathcal O
\end{align}
and
\begin{align}
	\omega(\varrho_1, \varrho_2):=\{x\in \Omega,\varrho_1<{\rm dist}(x,\partial\Omega)<\varrho_2\}\subset\mathcal O,\quad \varrho_1<\varrho_2<8\varrho.
\end{align}	
We then denote	 $\mathcal O_h:=\mathcal O\cap\Omega_h$,  $\omega_h(8\varrho):=\omega(8\varrho)\cap\Omega_h$, $\omega_h(\varrho_1, \varrho_2):=\omega(\varrho_1, \varrho_2)\cap\Omega_h$ (note that these sets are always non-empty for $h$ small enough).

Let $q\in \mathcal C(\Omega_h)$ satisfying $\|q\|_{L_h^\infty(\Omega_h)}\leq M$, $f\in \mathcal C(\Omega_h)$ satisfying $f=0$ in $\mathcal O_h$ and let $u\in\mathcal C(\overline\Omega_h)$ be such that
\begin{align}\label{uniq}
\begin{split}
	(-\Delta_h+q)u&=f, \quad {\rm in}~\Omega_h,\\
	u&=0, \quad  {\rm on}~\partial\Omega_h.
	\end{split}
\end{align}
We are now ready to show the result.
\begin{proposition}[Unique continuation estimate]\label{prou}
 Let $\Gamma$ be a non-empty open subset of $\partial \Omega$  satisfying  $\Gamma_h:=(\Gamma\cap\partial\Omega_h)\subset\partial_i^\pm\Omega_h$, $1\leq i\leq d$.
Then there exist constants  $0<h_0<1$, $\widetilde\varepsilon>0$, $C>0$, $\alpha_1>0$ and  $\alpha_2>0$ such that for all $0<h\leq h_0$,  $0<h\tau\leq\widetilde\varepsilon$ and all $u\in \mathcal C(\overline\Omega_h)$  satisfying (\ref{uniq}), we have
	\begin{align}\label{ee}
		\|u\|_{H^1_h(\omega_h(\varrho, 3\varrho))}\leq C\left(e^{-\alpha_1\tau}\|u\|_{H^1_h(\Omega_h)}+e^{\alpha_2\tau}|\partial_n u|_{L_h^2(\Gamma_h)}\right).
	\end{align}	
\end{proposition}	
In order to prove Proposition \ref{prou},  we first introduce a cut-off function $\chi$ satisfying $0\leq\chi\leq1$, $\chi\in C^\infty(\mathbb R)$ and
\begin{align}\label{chi}
		\chi(p)=
	\begin{cases}
		0,\quad p\leq\frac{1}{2},~p\geq8,\\
		1,\quad \frac{3}{4}\leq p\leq 7.
	\end{cases}
\end{align}	
In addition, for $x^{(1)}\in \Omega$, we set $B^1:=B(x^{(1)},r)=\{x\in\mathbb R^d,|x-x^{(1)}|<r\}$ and $B_h^1:=B^1\cap\Omega_h$. We shall begin to estimate $u$ in $B_h^1$ near the 	boundary part $\Gamma_h$.
\begin{lemma}\label{le42}
	Let $u$ be a solution to (\ref{uniq}). Then there exist $B^1_h=B(x^{(1)},r)\cap\Omega_h$ and $m_0\in(0,1)$ such that the following estimate holds:
	\begin{align}\label{buch7}
		\|u\|_{H_h^1(B_h^1)}\leq C\left(e^{-\frac{C_2\varepsilon_0}{m_0h}}\|u\|_{H_h^1(\Omega_h)}+
		|\partial_nu|_{L_h^2(\Gamma_h)}\right)^{m_0}\|u\|_{H_h^1(\Omega_h)}^{1-m_0},
	\end{align}
	for some positive constants $C$ and $C_2$. Here $\varepsilon_0$ is the constant given in Proposition \ref{Carleman}.
\end{lemma}
\begin{remark}
Compared with the continuous case, there is one extra term $e^{-\frac{C_2\varepsilon_0}{m_0h}}\|u\|_{H_h^1(\Omega_h)}$  due to the restriction $sh\leq\varepsilon_0$ of the discrete Carleman estimate in Proposition \ref{Carleman}.

\end{remark}
\begin{proof}[Proof of Lemma \ref{le42}]
	Let us choose $\delta>0$ and $x^{(0)}\in\mathbb R^d\backslash\overline\Omega$ such that
\begin{align}\label{e47}
	\delta<\frac{\varrho}{4},\quad \overline{B(x^{(0)},\delta)}\cap\overline\Omega=\emptyset,\quad B(x^{(0)},2\delta)\cap\Omega\neq\emptyset,\quad B(x^{(0)},4\delta)\cap\partial\Omega\subset\Gamma.
\end{align}	
That is, $x^{(0)}$ is an outer point of $\overline\Omega$ and is near $\Gamma$. We define the functions $\psi_0(x)$ and $\varphi_0(x)$ by:
\begin{align*}
	\psi_0(x)=|x-x^{(0)}|^2,\quad \varphi_0(x)=e^{-\frac{\lambda_0}{\delta^2}\psi_0(x)},
\end{align*}	
where $\lambda_0$ is the constant given in Proposition \ref{Carleman}.
Denote $w=\chi(\frac{\psi_0}{\delta^2})u$. Taking into account $u=0$ on $\partial\Omega_h$ and applying the discrete Carleman estimate (\ref{ce}), we obtain
\begin{align}\label{e48}
\begin{split}
	&s^3\|e^{s\varphi_0}w\|_{L^2_h(\Omega_h)}^2+s\sum_{k=1}^d\int_{\Omega_{h,k}^*}e^{2s\varphi_0}|D_kw|^2\\
	\leq &C\left(\|e^{s\varphi_0}(-\Delta_h+q)w\|_{L^2_h(\Omega_h)}^2+s\sum_{k=1}^d\int_{\partial_k\Omega_h}t_r^k(e^{2s\varphi_0})|\partial_nw|^2\right),
	\end{split}
\end{align}	
for $s_0\leq s\leq\frac{\varepsilon_0}{h}$.	 By (\ref{chi}) and (\ref{e47}), we have
\begin{align}\label{e49}
	s^3\|e^{s\varphi_0}w\|_{L^2_h(\Omega_h)}^2+s\sum_{k=1}^d\int_{\Omega_{h,k}^*}e^{2s\varphi_0}|D_kw|^2\geq & s^3\|e^{s\varphi_0}w\|_{L^2_h(E^1_h)}^2+s\sum_{k=1}^d\int_{(E^1_h)_k^*}e^{2s\varphi_0}|D_kw|^2\nonumber\\
	\geq &se^{2se^{-4\lambda_0}}\left(s^2\|u\|^2_{L^2_h(E^1_h)}+\|u\|^2_{\mathring H^1_h(E^1_h)}\right),
\end{align}
where $E^1_h:=\{x\in\Omega_h,\delta^2\leq\psi_0(x)\leq4\delta^2\}$.
In view of  $A_k^2\chi=\chi+\frac{h^2}{4}D_k^2\chi$ and
(\ref{uniq}), we get
\begin{align}\label{e410}
	(-\Delta_h+q)w=&-\sum_{k=1}^d\left[D_k(\sigma^kD_k\chi)A_k^2u+(D_kA_ku)A_k(\sigma^kD_k\chi)+(D_kA_k\chi) A_k(\sigma^kD_ku)\right]\nonumber\\
	&-\sum_{k=1}^d(\chi+\frac{h^2}{4}D_k^2\chi)D_k(\sigma^kD_ku)
	+q\chi u\nonumber\\
	\begin{split}
	=&-\sum_{k=1}^d\left[D_k(\sigma^kD_k\chi)A_k^2u+(D_kA_ku)A_k(\sigma^kD_k\chi)+(D_kA_k\chi) A_k(\sigma^kD_ku)\right]\\
	&-\frac{h^2}{4}\sum_{k=1}^d (D_k^2\chi) D_k(\sigma^kD_ku),
	\end{split}
\end{align}
for all $x\in\Omega_h$, where we have used
 $f\chi =0$ in $\Omega_h$ by noting  $f=0$ in $\mathcal O_h$ and (\ref{chi}).
Taking (\ref{e47}) into account, we see that $|x-x^{(0)}|>\delta$ for all $x\in \overline\Omega_h$, so that we obtain from (\ref{chi}) and (\ref{e410}) that
\begin{align*}
	\|e^{s\varphi_0}(-\Delta_h+q)w\|_{L^2_h(\Omega_h)}^2=&\|-e^{s\varphi_0}\sum_{k=1}^d\left[D_k(\sigma^kD_k\chi)A_k^2u+(D_kA_ku)A_k(\sigma^kD_k\chi)\right.
\\
&\left.+(D_kA_k\chi) A_k(\sigma^kD_ku)\right]-\frac{h^2}{4}e^{s\varphi_0}\sum_{k=1}^d (D_k^2\chi) D_k(\sigma^kD_ku)\|_{L^2_h(E^2_h)}^2\\
\leq &Ce^{2se^{-6\lambda_0}}\sum_{k=1}^d\left(\|A_k^2u\|_{L^2_h(\Omega_h)}^2+\|D_kA_ku\|_{L^2_h(\Omega_h)}^2+\|A_k(\sigma^kD_ku)\|_{L^2_h(\Omega_h)}^2\right)\\
&+Ce^{2se^{-6\lambda_0}}\|\sum_{k=1}^dD_k(\sigma^kD_ku)\|_{L^2_h(E^2_h)}^2,
\end{align*}
where $E^2_h:=	\{x\in\Omega_h,6\delta^2\leq\psi_0(x)\leq9\delta^2\}$. Recalling that
$f=0$ in $\mathcal O_h\supset\omega_h(8\varrho)\supset E^2_h$, we get
$$\|\sum_{k=1}^dD_k(\sigma^kD_ku)\|_{L^2_h(E^2_h)}^2=\|qu\|_{L^2_h(E^2_h)}^2\leq C\|u\|_{L^2_h(\Omega_h)}^2.$$
Furthermore, using Corollary \ref{co22} and (\ref{ak}), we have
\begin{align*}
	\|A_k^2u\|_{L^2_h(\Omega_h)}^2&\leq\int_{\Omega_h}A_k(|A_ku|^2)=\int_{\Omega_{h,k}^*}|A_ku|^2-\frac{h}{2}\int_{\partial_k\Omega_h}t_r^k(|A_ku|^2)
	\\&\leq \int_{\Omega_{h,k}^*}|A_ku|^2\leq\int_{\Omega_{h,k}^*}A_k(|u|^2)=\int_{\Omega_h}|u|^2.
\end{align*}
In the last equality, we have  used (\ref{ak}) by virtue of
 $u=0$ on $\partial_k\Omega_h$. Similarly, we have
$$\|D_kA_ku\|_{L^2_h(\Omega_h)}^2\leq\int_{\Omega_{h,k}^*}|D_ku|^2,\quad \|A_k(\sigma^kD_ku)\| _{L^2_h(\Omega_h)}^2\leq C\int_{\Omega_{h,k}^*}|D_ku|^2.$$
Then, we get
\begin{align}\label{e411}
	\|e^{s\varphi_0}(-\Delta_h+q)w\|_{L^2_h(\Omega_h)}^2\leq Ce^{2se^{-6\lambda_0}}\|u\|_{H^1_h(\Omega_h)}^2.
\end{align}
Taking (\ref{e47}) into account, we see that  $|x-x^{(0)}|>4\delta$ for all $x\in \partial\Omega_h\backslash\Gamma_h$ and $|x-x^{(0)}|>\delta$ for all $x\in \Gamma_h$, so that we obtain
\begin{align}\label{e412}
	s\sum_{k=1}^d\int_{\partial_k\Omega_h}t_r^k(e^{2s\varphi_0})|\partial_nw|^2&= s\int_{\Gamma_h}t_r^k(e^{2s\varphi_0})|\partial_nw|^2\nonumber\\
	&\leq se^{2se^{-\lambda_0}}\int_{\Gamma_h}|t_r^i(\sigma^iD_i(\chi u))n_i|^2\nonumber\\
	&\leq Cse^{2se^{-\lambda_0}}\int_{\Gamma_h}\left(|t_r^i(\sigma^iD_i\chi A_i u)|^2+|t_r^i(\sigma^iA_i\chi D_i u)|^2\right)\nonumber\\
	&\leq Cse^{2se^{-\lambda_0}}\int_{\Gamma_h}\left(|t_r^i(\sigma^i)t_r^i(A_i u)|^2+|t_r^i(\sigma^i D_i u)|^2\right)\nonumber\\
	&\leq Cse^{2se^{-\lambda_0}} |\partial_nu|^2_{L_h^2(\Gamma_h)},
\end{align}
since $\Gamma_h\subset\partial^\pm_i\Omega_h$ and $t_r^i(A_iv)=-\frac{h}{2}t_r^i(D_iv)n_i$  when $v=0$ on $\Gamma_h$.

Combining (\ref{e48}) with (\ref{e49}), (\ref{e411}) and (\ref{e412}), we obtain
\begin{align}\label{buch1}
	se^{2se^{-4\lambda_0}}\left(s^2\|u\|^2_{L^2_h(E^1_h)}+\|u\|^2_{\mathring H^1_h(E^1_h)}\right)\leq C\left(se^{2se^{-\lambda_0}} |\partial_nu|^2_{L_h^2(\Gamma_h)}+e^{2se^{-6\lambda_0}}\|u\|_{H^1_h(\Omega_h)}^2\right),
\end{align}
where $E^1_h:=\{x\in\Omega_h,\delta^2\leq\psi_0(x)\leq4\delta^2\}$.	

We can select $r>0$ and $x^{(1)}\in\Omega$ such that
\begin{align}\label{e413}
	{\rm dist}(x^{(1)},\Gamma_h)\geq4r, \quad  B^1_h:=B(x^{(1)},r)\cap \Omega_h\subset E_h^1.
\end{align}
This is possible because the second condition in (\ref{e47}) implies the existence of 	 $x^{(1)}\in\Omega$ such that $|x^{(1)}-x^{(0)}|<2\delta$. Therefore, for sufficiently small $r>0$, condition (\ref{e413}) is satisfied.  From (\ref{buch1}), we can obtain that    for $s_0\leq s\leq\frac{\varepsilon_0}{h}$, there exist constants $C_1,C_2>0$ such that
\begin{align}\label{buch2}
	\|u\|_{H^1_h(B^1_h)}^2\leq C\left(e^{C_1s}|\partial_nu|^2_{L_h^2(\Gamma_h)}+e^{-C_2s}\|u\|_{H^1_h(\Omega_h)}^2 \right).
\end{align}
If $\frac{1}{C_1+C_2}\log{\frac{\|u\|_{H^1_h(\Omega_h)}^2}{|\partial_nu|^2_{L_h^2(\Gamma_h)}}}\in [s_0,\frac{\varepsilon_0}{h}]$, then take $s=\frac{1}{C_1+C_2}\log{\frac{\|u\|_{H^1_h(\Omega_h)}^2}{|\partial_nu|^2_{L_h^2(\Gamma_h)}}}$ and we obtain
\begin{align}\label{buch3}
	\|u\|_{H^1_h(B^1_h)}^2\leq C|\partial_nu|^{2m_0}_{L_h^2(\Gamma_h)}\|u\|_{H^1_h(\Omega_h)}^{2(1-m_0)} ,\quad m_0=\frac{C_2}{C_1+C_2}.
\end{align}	
If $\frac{1}{C_1+C_2}\log{\frac{\|u\|_{H^1_h(\Omega_h)}^2}{|\partial_nu|^2_{L_h^2(\Gamma_h)}}}\leq s_0$, then
\begin{align*}
	\|u\|_{H^1_h(\Omega_h)}^2\leq C|\partial_nu|^{2}_{L_h^2(\Gamma_h)}.
\end{align*}	
Therefore,
\begin{align}\label{buch4}
	\|u\|_{H^1_h(B^1_h)}^2\leq\|u\|_{H^1_h(\Omega_h)}^2=\|u\|_{H^1_h(\Omega_h)}^{2m_0}\|u\|_{H^1_h(\Omega_h)}^{2(1-m_0)}\leq C|\partial_nu|^{2m_0}_{L_h^2(\Gamma_h)}\|u\|_{H_h^1(\Omega_h)}^{2(1-m_0)}.
\end{align}	
If $\frac{1}{C_1+C_2}\log{\frac{\|u\|_{H^1_h(\Omega_h)}^2}{|\partial_nu|^2_{L_h^2(\Gamma_h)}}}\geq \frac{\varepsilon_0}{h}$, then
\begin{align*}
	|\partial_nu|^{2}_{L_h^2(\Gamma_h)}\leq e^{-\frac{\varepsilon_0}{h}(C_1+C_2)}\|u\|_{H^1_h(\Omega_h)}^2 .
\end{align*}	
 In view of (\ref{buch2}), we have
\begin{align}\label{buch5}
	\|u\|_{H^1_h(B^1_h)}^2&\leq C\left(e^{C_1\frac{\varepsilon_0}{h}-\frac{\varepsilon_0}{h}(C_1+C_2)}\|u\|_{H^1_h(\Omega_h)}^2+e^{-C_2\frac{\varepsilon_0}{h}}\|u\|_{H^1_h(\Omega_h)}^2 \right)\nonumber\\
	&\leq Ce^{-C_2\frac{\varepsilon_0}{h}}\|u\|_{H^1_h(\Omega_h)}^{2m_0} \|u\|_{H^1_h(\Omega_h)}^{2(1-m_0)}.
\end{align}
 From (\ref{buch3})--(\ref{buch5}), we get (\ref{buch7}).
Thus, we complete the proof of Lemma \ref{le42}.
	
\end{proof}	
	
Next we extend the estimate from $B_h^1$ to $	\omega_h(\varrho, 3\varrho)$. To accomplish this, we adapt the techniques developed in \cite{robbiano1995fonction} to our discrete case. Let $B(x^{(j)},r), 2\leq j\leq N_0$, be a finite covering of $	\omega(\varrho, 3\varrho)$, such that $B_h(x^{(j)},r)=B(x^{(j)},r)\cap\Omega_h, 2\leq j\leq N_0$ is a finite covering of $\omega_h(\varrho, 3\varrho)=\omega(\varrho, 3\varrho)\cap\Omega_h$. We can assume that $x^{(j)}$ satisfies dist$(x^{(j)},\partial O)\geq 4r$. In the sequel, we assume without loss of generality that
\begin{align}
	B(x^{(j+1)},r)\subset B(x^{(j)},2r),
\end{align}
and we set $B^j=B(x^{(j)},r),	B^j_h=B(x^{(j)},r)\cap\Omega_h$, $2\leq j\leq N_0$.
\begin{lemma}\label{le43}
Let $u$ be a solution to (\ref{uniq}). Then there exist  $m\in(0,1)$ and $C, C_4>0$ such that the following estimate holds:
	\begin{align}\label{e415}
		\|u\|_{H_h^1(B_h^{j+1})}\leq C\left(e^{-\frac{C_4\varepsilon_0}{m h}}\|u\|_{H_h^1(\Omega_h)}+\|u\|_{H_h^1(B_h^{j})}\right)^{m}\|u\|_{H_h^1(\Omega_h)}^{1-m},\quad 1\leq j\leq N_0-1.
	\end{align}
	 Here $\varepsilon_0$ is the constant given in Proposition \ref{Carleman}.
\end{lemma}
\begin{proof}
	In order to prove (\ref{e415}),
we define the functions $\psi_j(x)$ and $\varphi_j(x)$ by:
\begin{align*}
	\psi_j(x)=|x-x^{(j)}|^2,\quad \varphi_j(x)=e^{-\frac{\lambda_0}{r^2}\psi_j(x)},
\end{align*}		
 where $\lambda_0$ is the constant given in Proposition \ref{Carleman}.
Moreover, we set $w=\chi(\frac{\psi_j}{r^2})u$. Taking  (\ref{chi}) and dist$(x^{(j)},\partial\Omega)\geq 4r$ into account, we have $w=\partial_nw=0$ on $\partial\Omega_h$. By applying the discrete Carleman estimate (\ref{ce}), we obtain, for $1\leq s_0\leq s\leq\frac{\varepsilon_0}{h}$,
\begin{align}
\begin{split}\label{e421a}
	s^3\|e^{s\varphi_j}w\|_{L^2_h(\Omega_h)}^2+s\sum_{k=1}^d\int_{\Omega_{h,k}^*}e^{2s\varphi_j}|D_kw|^2
	\leq C\|e^{s\varphi_j}(-\Delta_h+q)w\|_{L^2_h(\Omega_h)}^2.
	\end{split}
\end{align}
	In a similar way as the proof of (\ref{e48})  by (\ref{buch1}) in  Lemma \ref{le42}, we can prove
\begin{align*}
	se^{2se^{-5 \lambda_0}}\left(\|u\|^2_{L^2_h(E_h^4	)}+\|u\|^2_{\mathring H^1_h(E_h^4)}\right)\leq C\left(e^{2se^{-\frac{\lambda_0}{4}}}\|u\|_{H^1_h(E_h^5)}^2+e^{2se^{-6  \lambda_0}}\|u\|_{H^1_h(E_h^6)}^2\right),
\end{align*}		
where $E_h^4:=\{x\in\Omega_h,r^2\leq\psi_j(x)\leq5r^2\}$ , $E_h^5:=\{x\in\Omega_h, \frac{1}{4}r^2\leq\psi_j(x)\leq r^2\}$ and $E_h^6:=\{x\in\Omega_h, 6r^2\leq\psi_j(x)\leq 9 r^2\}$.	Let $E_h^7=\{x\in\Omega_h, \psi_j(x)\leq 5r^2\}$ and $E_h^8=\{x\in\Omega_h, \psi_j(x)\leq  r^2\}$. Hence, by (\ref{e421a}) and noting $E_h^7=E_h^8\cup E_h^4$, $E_h^5\subset E_h^8$, we obtain
\begin{align*}
	se^{2se^{-5\lambda_0}}\|u\|^2_{H^1_h(E_h^7)} &\leq se^{2se^{-5 \lambda_0}}\left(\|u\|^2_{H^1_h(E_h^8)}+\|u\|^2_{H^1_h(E_h^4)}\right)\\
&\leq C\left(se^{2se^{-5\lambda_0}}\|u\|^2_{H^1_h(E_h^8)}+e^{2se^{-\frac{\lambda_0}{4}}}\|u\|_{H^1_h(E_h^5)}^2+e^{2se^{-6\lambda_0}}\|u\|_{H^1_h(E_h^6)}^2\right)\\
	&\leq C\left(e^{2se^{-\frac{\lambda_0}{4}}}\|u\|_{H^1_h(E_h^8)}^2+e^{2se^{-6\lambda_0}}\|u\|_{H^1_h(E_h^6)}^2\right).
\end{align*}

Thus we obtain
\begin{align*}
	\|u\|^2_{H^1_h(E_h^7)}\leq C\left(e^{C_3s}\|u\|_{H^1_h(E_h^8)}^2+e^{-C_4s}\|u\|_{H^1_h(\Omega_h)}^2\right),
\end{align*}	
for $s_0\leq s\leq\frac{\varepsilon_0}{h}$ and  some constants $C_3,C_4>0$.
Similarly, minimizing the right-hand side with respect to $s$, for some $m\in(0,1)$, we have
\begin{align*}
	\|u\|_{H^1_h(E_h^7)}\leq C\left(e^{-\frac{C_4\varepsilon_0}{m h}}\|u\|_{H_h^1(\Omega_h)}+\|u\|_{H^1_h(E_h^8)}\right)^{m}\|u\|_{H^1_h(\Omega_h)}^{(1-m)}.
\end{align*}	
Since $$B^{j+1}_h\subset E_h^7,\quad \quad E_h^8\subset B^j_h,$$
we obtain (\ref{e415}). This completes the proof of  Lemma \ref{le43}.

\end{proof}

\begin{lemma}\label{le44}
Let $u$ be a solution to (\ref{uniq}). Then there exists  constant $C>0$  such that
	\begin{align}
		\|u\|_{H_h^1(B_h^{j})}\leq C\left(e^{-\frac{C_4\varepsilon_0}{m h}}\|u\|_{H_h^1(\Omega_h)}+\|u\|_{H_h^1(B_h^{1})}\right)^{m^j}\|u\|_{H_h^1(\Omega_h)}^{1-m^j},\quad 2\leq j\leq N_0.
	\end{align}
	Here $m\in(0,1)$ is the constant given in Lemma \ref{le43}.
\end{lemma}
\begin{proof}	
Put
$$\alpha_j=	\|u\|_{H_h^1(B_h^{j})},\quad A=e^{-\frac{C_4\varepsilon_0}{m h}}\|u\|_{H_h^1(\Omega_h)}, \quad B=C^\frac{1}{1-m}\|u\|_{H_h^1(\Omega_h)},$$
by (\ref{e415}) we have
$$\alpha_{j+1}\leq B^{1-m}(\alpha_j+A)^m,\quad 1\leq j\leq N_0-1.$$
Then,  applying Lemma 4 in \cite{lebeau1994controle}, we obtain, for all $\mu\in(0,m^j]$,
$$\alpha_j\leq 2^{\frac{1}{1-m}}B^{1-\mu}(\alpha_1+A)^{\mu},\quad 2\leq j\leq N_0.$$
This completes the proof of the lemma.
\end{proof}

\begin{proof}[Proof of Proposition \ref{prou}]
By Lemma \ref{le44} and  Young's inequality, we easily obtain
$$\|u\|_{H_h^1(B_h^{j})}\leq C\left(\varepsilon^p\|u\|_{H_h^1(\Omega_h)}+\varepsilon^{-p'}	\left(e^{-\frac{C_4\varepsilon_0}{m h}}\|u\|_{H_h^1(\Omega_h)}+\|u\|_{H_h^1(B_h^{1})}\right)\right),$$
for all $\varepsilon>0, 2\leq j\leq N_0$. Here
$$p=\frac{1}{1-m^j},\quad p'=\frac{1}{m^j}.$$
Selecting $\varepsilon=e^{-\frac{C_5}{p}\tau}$, $C_5>0,\tau>0$, we have
\begin{align*}
\|u\|_{H_h^1(B_h^{j})}\leq C\left(e^{-C_5\tau}\|u\|_{H_h^1(\Omega_h)}+e^{C_6(j)\tau-\frac{C_4\varepsilon_0}{m h}}\|u\|_{H_h^1(\Omega_h)}+e^{C_6(j)\tau}	\|u\|_{H_h^1(B_h^{1})}\right),
\end{align*}	
where $C_6(j):=\frac{C_5p'}{p}$. Then, for $h\tau\leq\frac{C_4\varepsilon_0}{(C_5+C_6(j))m}$ such that $C_6(j)\tau-\frac{C_4\varepsilon_0}{m h}\leq -C_5\tau$, we get
\begin{align}\label{e418}
	\|u\|_{H_h^1(B_h^{j})}\leq C\left(e^{-C_5\tau}\|u\|_{H_h^1(\Omega_h)}+e^{C_6(j)\tau}	\|u\|_{H_h^1(B_h^{1})}\right).
\end{align}

Similarly, we obtain from Lemma	\ref{le42} and Young's inequality
\begin{align*}
	\|u\|_{H_h^1(B_h^1)}\leq C\left(\varepsilon^{p_0}\|u\|_{H_h^1(\Omega_h)}+\varepsilon^{-p_0'} \left(e^{-\frac{C_2\varepsilon_0}{m_0h}}\|u\|_{H_h^1(\Omega_h)}+
		|\partial_nu|_{L_h^2(\Gamma_h)}\right)\right),
\end{align*}
for all $\varepsilon>0$. Here
$$p_0=\frac{1}{1-m_0},\quad p_0'=\frac{1}{m_0}.$$	
Selecting $\varepsilon=e^{-\frac{C_5+C_6(j)}{p_0}\tau}$ and setting $C_7(j)=\frac{(C_5+C_6(j))p_0'}{p_0}$, 	we obtain
\begin{align*}
\|u\|_{H_h^1(B_h^1)}\leq C\left(e^{-(C_5+C_6(j))\tau}\|u\|_{H_h^1(\Omega_h)}+e^{C_7(j)\tau-\frac{C_2\varepsilon_0}{m_0h}}\|u\|_{H_h^1(\Omega_h)}+e^{C_7(j)\tau} |\partial_nu|_{L_h^2(\Gamma_h)}\right).	
\end{align*}
Choosing $h\tau\leq\frac{C_2\varepsilon_0}{(C_5+C_6(j)+C_7(j))m_0}$ such that $C_7(j)\tau-\frac{C_2\varepsilon_0}{m_0h}\leq -(C_5+C_6(j))\tau$, we have
\begin{align}\label{e419}
	\|u\|_{H_h^1(B_h^1)}\leq C\left(e^{-(C_5+C_6(j))\tau}\|u\|_{H_h^1(\Omega_h)}+e^{C_7(j)\tau} |\partial_nu|_{L_h^2(\Gamma_h)}\right).
\end{align}	
By (\ref{e418}) and (\ref{e419}), we have
\begin{align}\label{e420}
	\|u\|_{H_h^1(B_h^j)}\leq Ce^{-C_5\tau}\|u\|_{H_h^1(\Omega_h)}+e^{C(j)\tau} |\partial_nu|_{L_h^2(\Gamma_h)},
\end{align}	
for some positive constant $C(j)$ and $\tau>0$ satisfying $h\tau\leq\min\{\frac{C_4\varepsilon_0}{(C_5+C_6(j))m},\frac{C_2\varepsilon_0}{(C_5+C_6(j)+C_7(j))m_0}\}:=\varepsilon_j$. Setting $\widetilde\varepsilon:=\min_{j\in\{2,\dots,N_0\}}\varepsilon_j$,
addition of inequalities (\ref{e420}) for $j\in\{2,\dots,N_0\}$ yields the final inequality (\ref{ee}). 	
Thus, the proof of Proposition \ref{prou} is complete.

\end{proof}

\section{Stability estimate for the discrete Carder\'on problem}
 In this section, we focus on the proof of Theorem \ref{theo}.

Similar to the continuous case, the proof of the stability estimate is based on the existence of CGO solutions. The following results state their existence.

\begin{lemma}[{\cite[Theorem 4.4]{ervedoza2011uniform}}]\label{CGO}
Let $M>0$. For all $q\in C(\Omega_h)$ satisfying $\|q\|_{L_h^\infty(\Omega_h)}\leq M$ and  Assumption \ref{assq}, there exist constants $a_0>0$ and $c>0$ that depend on $M$ such that  $\forall\eta\in\mathbb C^d$ with $\eta\cdot\eta=0$, if $a:=|{\rm Im}\eta|$ verifies $a_0\leq a\leq c\min\{\epsilon_d^{-1},h^{-\frac{2}{3}}\}$, there exists  $u\in\mathcal C(\overline\Omega_h)$ a solution of
$$-\Delta_hu+qu=0, \quad {\rm on}~\Omega_h,$$
that satisfies
\begin{align}\label{e51}
u(x)=e^{{\rm i}\eta\cdot x}(1+r(x)),\quad {\rm on}~\overline\Omega_h,
\end{align}
with
\begin{align}\label{reme}
\|r\|_{\mathring H_h^1(\Omega_h)}+a\|r\|_{L_h^2(\overline\Omega_h)}\leq C(1+a^2\epsilon_a+a^4h^2).
\end{align}
Moreover, the solution $u$ satisfies
\begin{align}\label{e52}
	\|u\|_{H^1_h(\Omega_h)}\leq Ce^aa^2,
\end{align}
where $C$ only depends on $m, a_0$ and $c$.
\end{lemma}

\begin{remark}
As it mentioned in (4.21) in \cite{ervedoza2011uniform}, the solution $u=e^{\eta\cdot x}+e^{\mathcal R(\eta)\cdot x}r(x)$ in \cite{ervedoza2011uniform} can be rewritten as (\ref{e51}) in our paper.
\end{remark}

\begin{proof}[Proof of Theorem \ref{theo}]
Let $q_1\geq0$ and $q_2\geq0$ in $\mathcal C(\Omega_h)$ be two potentials satisfying	Assumption \ref{assq} and $q_1=q_2$ in $\mathcal O_h=\mathcal O\cap\Omega_h$, where $\mathcal O\subset\Omega$ is a neighborhood of $\partial\Omega$ and  $\|q\|_{L_h^\infty(\Omega_h)}\leq M$. For all $	\xi\in \hat{\mathcal  K}_h=\mathbb Z^d\cap[0,N]^d$ satisfying $\pi|\xi|\leq a$,
we set $$\eta_1:=-\pi\xi+\sqrt{a^2-\pi^2|\xi|^2}\zeta_1+{\rm i}a\zeta_2,\quad \eta_2:=-\pi\xi-\sqrt{a^2-\pi^2|\xi|^2}\zeta_1-{\rm i}a\zeta_2,$$
where $a_0\leq a\leq c\min\{\epsilon_d^{-1},\epsilon_a^{-1},h^{-\frac{2}{3}}\}$ and $ \zeta_1,\zeta_2\in\mathbb R^d$ satisfying $\xi\cdot\zeta_1=\xi\cdot\zeta_2=\zeta_1\cdot\zeta_2=0 , | \zeta_1|=|\zeta_2|=1.$ It follows from Lemma \ref{CGO} that there exist $r_1$ and $r_2$, such that $u_1:=e^{{\rm i}\eta_1\cdot x}(1+r_1(x))$ and $u_2:=e^{{\rm i}\eta_2\cdot x}(1+r_2(x))$ are solutions of
\begin{align*}
	&-\Delta_hu_1+q_1u_1=0\quad {\rm in }~\Omega_h,\\
	&-\Delta_hu_2+q_2u_2=0\quad {\rm in }~\Omega_h,
\end{align*}
respectively. 	
Let $v\in \mathcal C(\overline\Omega_h)$ be the solution to the following problem
\begin{align*}
	(-\Delta_h+q_1)v&=0 ~\quad{\rm in} ~\Omega_h,\\
	v&=u_2\quad{\rm on}~~\partial\Omega_h.
\end{align*}
Then, by defining $u=v-u_2\in \mathcal C(\overline\Omega_h)$, we get
\begin{align}\label{a3}
	\begin{split}
		(-\Delta_h+q_1)u&=(q_2-q_1) u_2 \quad\mbox{in} ~~\Omega_h,\\
	u&=0\quad\mbox{on}~~\partial\Omega_h.
	\end{split}
\end{align}	
To proceed, we introduce a cut-off function $\chi_1\in C_0^\infty(\Omega)$ satisfying $0\leq\chi_1\leq1$ and
\begin{align}\label{m6}
	\chi_1(x)=
	\begin{cases}
		0,\quad x\in\omega(\frac{3}{2}\varrho),\\
		1,\quad x\in \overline\Omega\backslash\omega(\frac{5}{2}\varrho),
	\end{cases}
\end{align}
where $\omega(\frac{3}{2}\varrho)$ and $\omega(\frac{5}{2}\varrho)$ are defined in Section 4.
If we put $\widetilde u=\chi_1 u$, it follows from $A_k^2\chi_1=\chi_1+\frac{h^2}{4}D_k^2\chi_1$  and (\ref{a3}) that
\begin{align}\label{e55}
	(-\Delta_h+q_1)\widetilde u=&-\sum_{k=1}^d\left[D_k(\sigma^kD_k\chi_1)A_k^2u+(D_kA_ku)A_k(\sigma^kD_k\chi_1)+(D_kA_k\chi_1) A_k(\sigma^kD_ku)\right]\nonumber\\
	&-\sum_{k=1}^d(\chi_1+\frac{h^2}{4}D_k^2\chi_1)D_k(\sigma^kD_ku)
	+q_1\chi_1 u\nonumber\\
	\begin{split}
	=&-\sum_{k=1}^d\left[D_k(\sigma^kD_k\chi_1)A_k^2u+(D_kA_ku)A_k(\sigma^kD_k\chi_1)+(D_kA_k\chi_1) A_k(\sigma^kD_ku)\right]\\
	&-\frac{h^2}{4}\sum_{k=1}^d (D_k^2\chi_1) D_k(\sigma^kD_ku)+(q_2-q_1)u_2,
	\end{split}
\end{align}
for all $x\in\Omega_h$,
where we have used $\chi_1(q_2-q_1)=q_2-q_1$ in $\Omega_h$ by noting $q_1=q_2$ in ${\mathcal O_h}$ and (\ref{m6}).	
Multiplying  (\ref{e55}) by $u_1$ and  integrating  over $\Omega_h$, we obtain	
\begin{align}\label{e56}
	\begin{split}
		\int_{\Omega_h}(q_2-q_1)u_1u_2=\int_{\Omega_h}u_1\sum_{k=1}^d&\left[D_k(\sigma^kD_k\chi_1)A_k^2u+(D_kA_ku)A_k(\sigma^kD_k\chi_1)\right.\\
		&\left.+(D_kA_k\chi_1) A_k(\sigma^kD_ku)+\frac{h^2}{4} (D_k^2\chi_1)D_k(\sigma^kD_ku)\right],
	\end{split}
\end{align}	
where we have used 	(\ref{e23}) by noting $\widetilde u=\partial_n\widetilde u=0$ on $\partial\Omega_h$. In order to simplify the notations, we denote the right-hand side of (\ref{e56})
as $F$. Noting that $u$ satisfies the conditions in Proposition \ref{prou}, we apply (\ref{ee}) and (\ref{m6}) to get
\begin{align*}
	|F|^2\leq &C\|u_1\|_{L^2_h(\Omega_h)}^2\sum_{k=1}^d\int_{\omega_h(\varrho,3\varrho)}\left(|A_k^2u|^2+|D_kA_ku|^2+|A_k(\sigma^kD_ku)|^2+\frac{h^2}{4}|D_k(\sigma^kD_ku)|^2\right),
\end{align*}
for $0<h<h_0<1$.
Similar to the discussions in Section 4, we have
\begin{align*}
\|A_k^2u\|_{L^2_h(\omega_h(\varrho,3\varrho))}^2&\leq \|u\|_{L^2_h(\omega_h(\varrho,3\varrho))}^2,\\
 \|D_kA_ku\|_{L^2_h(\omega_h(\varrho,3\varrho))}^2&\leq\|u\|_{\mathring H_h^1(\omega_h(\varrho,3\varrho))}^2,\\
	\|A_k(\sigma^kD_ku)\| _{L^2_h(\omega_h(\varrho,3\varrho))}^2&\leq C\|u\|_{\mathring H_h^1(\omega_h(\varrho,3\varrho))}^2,\\
	\frac{h^2}{4}\|D_k(\sigma^kD_ku)\| _{L^2_h(\omega_h(\varrho,3\varrho))}^2&\leq C\|u\|_{\mathring H_h^1(\omega_h(\varrho,3\varrho))}^2.
\end{align*}
In the last inequality, we used $\frac{h^2}{4}|D_k(\sigma^kD_ku)|^2\leq A_k(|\sigma^kD_ku|^2)$.
Thus, we get
\begin{align}\label{e57}
	|F|\leq C\|u_1\|_{L^2_h(\Omega_h)}\|u\|_{H_h^1(\omega_h(\varrho,3\varrho))}
\end{align}
Combining (\ref{e57}) with (\ref{e56}) and (\ref{ee}), we obtain
\begin{align}\label{e58}
	\left|\int_{\Omega_h}(q_2-q_1)u_1u_2\right|\leq C\|u_1\|_{L^2_h(\Omega_h)}\left(e^{-\alpha_1\tau}\|u\|_{H^1_h(\Omega_h)}+e^{\alpha_2\tau}|\partial_n u|_{L^2_h(\Gamma_h)}\right),
\end{align}
where $\alpha_1,\alpha_2>0$ and $0<\tau\leq\frac{\widetilde\varepsilon}{h}$.

Applying Proposition \ref{regu} and noting that $u$ satisfies (\ref{a3}), we have
\begin{align*}
	|\partial_n u|_{L^2_h(\Gamma_h)}&\leq |\partial_n u|^{\frac{1}{2}}_{H^{-\frac{1}{2}}_h(\Gamma_h)}|\partial_n u|^{\frac{1}{2}}_{H^{\frac{1}{2}}_h(\Gamma_h)} =|(\Lambda_h[q_1]-\Lambda_h[q_2])(u_2)|^{\frac{1}{2}}_{H^{-\frac{1}{2}}_h(\Gamma_h)}|\partial_n u|^{\frac{1}{2}}_{H^{\frac{1}{2}}_h(\Gamma_h)}\\
	  &\leq \delta^\frac{1}{2}|u_2|^\frac{1}{2}_{H_h^{\frac{1}{2}}(\partial\Omega_h)}|\partial_n u|^{\frac{1}{2}}_{H^{\frac{1}{2}}_h(\Gamma_h)}\leq \delta^\frac{1}{2}\|u_2\|^\frac{1}{2}_{H_h^{1}(\Omega_h)}\|u\|^\frac{1}{2}_{H_h^{2}(\Omega_h)}\\
	&\leq C\delta^\frac{1}{2} \|u_2\|^\frac{1}{2}_{H_h^{1}(\Omega_h)}\|u_2\|^\frac{1}{2}_{L^{2}_h(\Omega_h)},
\end{align*}
where we abbreviate $\delta:=\|\Lambda_h[q_1]-\Lambda_h[q_2]\|_{\mathcal L_h(\Gamma_h)}$.
It follows from (\ref{e52}) that
\begin{align*}
	|\partial_n u|_{L^2_h(\Gamma_h)}\leq  C\delta^\frac{1}{2} e^a a^2,
\end{align*}
 and
 \begin{align*}
\|u_1\|_{L^2_h(\Omega_h)}\leq Ce^a a^2.
\end{align*}
Similarly,  we   apply Assumption \ref{as2} and (\ref{e52}) to get
\begin{align*}
\|u\|_{H^1_h(\Omega_h)}\leq C\|(q_2-q_1)u_2\|_{L^2_h(\Omega_h)}\leq Ce^a a^2.	
\end{align*}

Thus, we have	
	\begin{align}\label{e59}
	\left|\int_{\Omega_h}(q_2-q_1)u_1u_2\right|\leq C	e^{2a}a^4\left(e^{-\alpha_1\tau}
	+\delta^{\frac{1}{2}}e^{\alpha_2\tau}\right),
	\end{align}
for all $0<\tau\leq\frac{\widetilde\varepsilon}{h}$ and $\alpha_1>0,\alpha_2>0,a_0\leq a\leq c\min\{\epsilon_d^{-1},\epsilon_a^{-1},h^{-\frac{2}{3}}\}$.	
Then we choose a constant $\widetilde{\gamma} >0$ sufficiently large and set $\tau=a\widetilde\gamma$ so that
\begin{align}\label{buch11}
	e^{2a}a^4e^{-\alpha_1\tau}\leq \frac{C}{a},\quad {\rm and}\quad
	e^{2a}e^{\alpha_2\tau}\leq e^{\alpha_3 a},
\end{align}
for some constants $C>0$ and $\alpha_3>0$.  It is easy to see that	$0<\tau\leq\frac{\widetilde\varepsilon}{h}$ implies that $a\leq \frac{\widetilde\varepsilon}{\widetilde\gamma}h^{-1}$.
Furthermore,  by (\ref{e59}), (\ref{buch11}) and noting $u_1=e^{{\rm i}\eta_1\cdot x}(1+r_1(x))$, $u_2=e^{{\rm i}\eta_2\cdot x}(1+r_2(x))$, we obtain
\begin{align}\label{e511}
	\left|\mathcal F_h(q_2-q_1)(\xi)\right|&\leq \frac{C}{a}+Ca^4e^{\alpha_3a}\delta^{\frac{1}{2}}+\left|\int_{\Omega_h}(q_2-q_1)e^{-2\pi{\rm i}\xi\cdot x}
(r_1+r_2+r_1r_2)\right| \nonumber\\
&\leq C\left(\frac{1}{a}+a^4e^{\alpha_3a}\delta^{\frac{1}{2}}+\|r_1\|_{L^2_h(\Omega_h)}+\|r_2\|_{L^2_h(\Omega_h)}+\|r_1\|_{L^2_h(\Omega_h)}\|r_2\|_{L^2_h(\Omega_h)}\right),
\end{align}	
In addition, it follows from 	(\ref{reme}) that $\|r_1\|_{L^2_h(\Omega_h)}$, $\|r_2\|_{L^2_h(\Omega_h)}$ are bounded and
$$\|r_1\|_{L^2_h(\Omega_h)}+\|r_2\|_{L^2_h(\Omega_h)}+\|r_1\|_{L^2_h(\Omega_h)}\|r_2\|_{L^2_h(\Omega_h)}\leq C\left(\frac{1}{a}+a\varepsilon_a+a^3h^2\right).$$
We conclude from (\ref{e511}) that
\begin{align}\label{e512}
	\left|\mathcal F_h(q_2-q_1)(\xi)\right|\leq C\left(a^4e^{\alpha_3a}\delta^{\frac{1}{2}}+\frac{1}{a}+a\varepsilon_a+a^3h^2\right),
\end{align}
for $\alpha_3>0$ , $a_0\leq a\leq\widetilde c\min\{\epsilon_d^{-1},\epsilon_a^{-1},h^{-\frac{2}{3}}\}$, $\widetilde c:=\min\{\frac{\widetilde\varepsilon}{\widetilde\gamma},c\}$ and $	\xi\in \hat{\mathcal  K}_h$ satisfying $\pi|\xi|\leq a$.

Setting $\widetilde\mu=\widetilde c^{-1}\max\{\varepsilon_a^{\frac{1}{2}},h^{\frac{1}{2}},\varepsilon_d\}$, we always have $\frac{1}{\widetilde\mu}\leq\widetilde c\min\{\epsilon_d^{-1},\epsilon_a^{-1},h^{-\frac{2}{3}}\}$ and a simple study shows that, setting $a=\frac{1}{\widetilde\mu}$,
$$
\frac{1}{a}+a\varepsilon_a+a^3h^2\leq C\widetilde\mu.$$

Therefore, 	if $\delta=0$, we set $a=\frac{1}{\widetilde\mu}$, then for all $	\xi\in \hat{\mathcal  K}_h$ satisfying $\pi|\xi|\leq \frac{1}{\widetilde\mu}$, we have
\begin{align*}
	\left|\mathcal F_h(q_2-q_1)(\xi)\right|\leq C\widetilde\mu.
\end{align*}

If $\delta$ is small enough, for example, for $-\ln\delta\geq\frac{2(\alpha_3+3)}{\widetilde\mu}	$, such that
$$\delta^{\frac{1}{2}}\frac{1}{\widetilde\mu^4}e^{\frac{\alpha_3}{\widetilde\mu}}\leq\widetilde\mu,$$
taking  	$a=\frac{1}{\widetilde\mu}$ in (\ref{e512}), we obtain,
for all $	\xi\in \hat{\mathcal  K}_h$ satisfying $\pi|\xi|\leq \frac{1}{\widetilde\mu}$,
\begin{align*}
\left|\mathcal F_h(q_2-q_1)(\xi)\right|\leq C\widetilde\mu.	
\end{align*}	

If $-\ln\delta\in (2(\alpha_3+3)a_0,\frac{2(\alpha_3+3)}{\widetilde\mu})	$, taking $a=\frac{-\ln\delta}{2(\alpha_3+3)}$ in (\ref{e512}), we obtain,
for all $	\xi\in \hat{\mathcal  K}_h$ satisfying $\pi|\xi|\leq \frac{-\ln\delta}{2(\alpha_3+3)}$,
\begin{align*}
\left|\mathcal F_h(q_2-q_1)(\xi)\right|&\leq C\left(\delta^{\frac{3}{2(\alpha_3+3)}	}|\ln\delta|^4+\frac{1}{|\ln\delta|}+|\ln\delta|\varepsilon_a+|\ln\delta|^3h^2\right)\\
&\leq \frac{C}{|\ln\delta|},
\end{align*}		
where we have used  $\widetilde\mu=\widetilde c^{-1}\max\{\varepsilon_a^{1/2},h^{1/2},\varepsilon_d\}\leq \frac{2(\alpha_3+3)}{|\ln\delta|}$.

Combining these cases, we obtain that, if $\delta\leq e^{-2(\alpha_3+3)a_0}$, setting	
$$\mu=\max\{\widetilde\mu,\frac{2(\alpha_3+3)}{|\ln\delta|}\},$$
for all $	\xi\in \hat{\mathcal  K}_h$ satisfying $\pi|\xi|\leq \frac{1}{\mu}$,
\begin{align}\label{e513}
	\left|\mathcal F_h(q_2-q_1)(\xi)\right|\leq C\mu.	
\end{align}	

In order to estimate $|q_1-q_2|_{H_h^{-r}(\Omega_h)}$, we
take $\rho\in(0,\frac{1}{\pi\mu})$ to be chosen and use (\ref{e513}) to get 	
\begin{align*}
	|q_1-q_2|^2_{H_h^{-r}(\Omega_h)}=&\sum_{\xi\in\hat{\mathcal  K}_h}\left|\mathcal F_h(q_2-q_1)(\xi)\right|^2(1+|\xi|^2)^{-r}\\
	=&\sum_{|\xi|<\rho,\xi\in\hat{\mathcal  K}_h}\left|\mathcal F_h(q_2-q_1)(\xi)\right|^2(1+|\xi|^2)^{-r}\\
	&+\sum_{|\xi|>\rho,\xi\in\hat{\mathcal  K}_h}\left|\mathcal F_h(q_2-q_1)(\xi)\right|^2(1+|\xi|^2)^{-r}\\
	\leq&C(\rho^d\mu^2+\rho^{-2r}).
\end{align*}
Setting $\rho=\pi^{-1}\mu^{\frac{-2}{d+2r}}$, which	 is indeed smaller than $\frac{1}{\pi\mu}$, we have
\begin{align*}
	|q_1-q_2|^2_{H_h^{-r}(\Omega_h)}&\leq C\mu^{\frac{4r}{d+2r}}\\
       &\leq C\max\left\{\epsilon_d^{\frac{4r}{d+2r}}, \epsilon_a^{\frac{2r}{d+2r}}, h^{\frac{2r}{d+2r}}, \left|\ln\left(\|\Lambda_h[q_1]-\Lambda_h[q_2]\|_{\mathcal L_h(\Gamma_h)}\right)\right|^{-\frac{4r}{d+2r}}\right\},
\end{align*}
which completes the proof of Theorem \ref{theo}.

\end{proof}

\begin{appendix}
\section{Proofs of intermediate results}	
\subsection{Proof of Lemma \ref{lerhs}}
By Assumption \ref{simg} and the second equality in Lemma  \ref{le34}, we have
\begin{align}\label{A1}
\|rD_kA_k\rho D_k\sigma^k \tau_kD_kv\|_{L^2_h(\Omega_h)}^2	&\leq\int_{\Omega_h}|rD_kA_k\rho|^2|\tau_kD_kv|^2\leq  C_{\lambda,\mathcal R}s^2\int_{\Omega_h}|\tau_kD_kv|^2 \nonumber \\
&=C_{\lambda,\mathcal R}s^2\int_{\tau_k(\Omega_h)}|D_kv|^2\leq C_{\lambda,\mathcal R}s^2\int_{\Omega_{h,k}^*}|D_kv|^2.
\end{align}
Similarly, we have
\begin{align}
	\|rD_kA_k\rho D_k\sigma^k \tau_{-k}D_kv\|_{L^2_h(\Omega_h)}^2\leq C_{\lambda,\mathcal R}s^2\int_{\Omega_{h,k}^*}|D_kv|^2.
\end{align}
By Assumption \ref{simg}, Corollary \ref{co22} and (\ref{ak}), we  observe that
\begin{align*}
		\|rD_k^2\rho D_k\sigma^k D_kA_kv\|^2_{L^2_h(\Omega_h)}&\leq \int_{\Omega_h}|rD_k^2\rho|^2A_k(|D_kv|^2)\\
		&=\int_{\Omega_{h,k}^*}A_k(|rD_k^2\rho|^2)|D_kv|^2-\frac{h}{2}\int_{\partial_k\Omega_h}|rD_k^2\rho|^2t_r^k(|D_kv|^2)\\
		&\leq\int_{\Omega_{h,k}^*}A_k(|rD_k^2\rho|^2)|D_kv|^2,
\end{align*}
which, by the third equality in Lemma \ref{le34}, yields
\begin{align}
	\|rD_k^2\rho D_k\sigma^k D_kA_kv\|^2_{L^2_h(\Omega_h)}\leq C_{\lambda,\mathcal R}s^4\int_{\Omega_{h,k}^*}|D_kv|^2.
\end{align}
We also find
\begin{align}
	\|rA_kD_k\rho A_kD_kv\|^2_{L^2_h(\Omega_h)}\leq C_{\lambda,\mathcal R}s^2\int_{\Omega_{h,k}^*}|D_kv|^2.
\end{align}
We note that
\begin{align*}
	\|A_k^2v\|^2_{L^2_h(\Omega_h)}\leq \int_{\Omega_h} A_k(|A_kv|^2)=\int_{\Omega_{h,k}^*}|A_kv|^2-\frac{h}{2}\int_{\partial_k\Omega_h}t_r^k(|A_kv|^2)\leq \int_{\Omega_{h,k}^*}|A_kv|^2\leq\int_{\Omega_h}|v|^2,
\end{align*}
by Corollary \ref{co22}, (\ref{ak}) and since $v|_{\partial\Omega_h}=0$. Then by the second and third equality in  Lemma \ref{le34}, we have
\begin{align}\label{A5}
\|(rD_kA_k\rho D_k\sigma^k+hO(1)rD_k^2\rho)A_k^2v\|^2_{L^2_h(\Omega_h)}\leq 	C_{\lambda,\mathcal R}s^2(1+(sh)^2)\|v\|^2_{L^2_h(\Omega_h)}.
\end{align}
Similarly, since $\Delta_\sigma\varphi$ is bounded, estimates (\ref{A1})-(\ref{A5}) yield the result.

\subsection{Proof of Lemma \ref{le37}}

From the definitions of $A_1v$ and $B_1v$  in (\ref{abg}) we have
\begin{align*}
	(A_1v,B_1v)_{\Omega_h}&=\sum_{j,k=1}^d\int_{\Omega_h}2r^2A_k^2\rho D_k(\sigma^kD_kv)\sigma^jA_jD_j\rho A_jD_jv\nonumber\\
	&:=\sum_{j,k=1}^d\int_{\Omega_h}2\gamma_1^kD_k(\sigma^kD_kv)\beta_1^j\sigma^jA_jD_jv\nonumber\\
	&:=\sum_{j,k=1}^dI_{jk},
\end{align*}
where $\gamma_1^k=rA_k^2\rho$ and $\beta_1^j=rA_jD_j\rho$.

\subsubsection{Computation of $I_{kk}$}
By Lemma \ref{le21} and (\ref{dk}), we have
\begin{align}\label{A6}
	I_{kk}=&\int_{\Omega_h}2\gamma_1^k\beta_1^k\sigma^kA_k\sigma^kD_k^2vD_kA_kv+\int_{\Omega_h}2\gamma_1^k\beta_1^k\sigma^kD_k\sigma^k|A_kD_kv|^2\nonumber\\
	=&\int_{\Omega_h}\gamma_1^k\beta_1^k\sigma^kA_k\sigma^kD_k(|D_kv|^2)+\int_{\Omega_h}2\gamma_1^k\beta_1^k\sigma^kD_k\sigma^k|A_kD_kv|^2\nonumber\\
	\begin{split}
	=&-\int_{\Omega_{h,k}^*}D_k(\gamma_1^k\beta_1^k\sigma^kA_k\sigma^k)|D_kv|^2+\int_{\partial_k\Omega_h}\gamma_1^k\beta_1^k\sigma^kA_k\sigma^kt_r^k(|D_kv|^2)n_k\\
	&+\int_{\Omega_h}2\gamma_1^k\beta_1^k\sigma^kD_k\sigma^k|A_kD_kv|^2.
	\end{split}
\end{align}
The discrete integration by parts with respect to the average operator $A_k$ (\ref{ak}) yields
\begin{align}
\begin{split}
	&-\int_{\Omega_{h,k}^*}D_k(\gamma_1^k\beta_1^k\sigma^kA_k\sigma^k)|D_kv|^2\\=&\int_{\Omega_{h,k}^*}D_k(\gamma_1^k\beta_1^k\sigma^kA_k\sigma^k)|D_kv|^2-2\int_{\Omega_h}A_k(D_k(\gamma_1^k\beta_1^k\sigma^kA_k\sigma^k)|D_kv|^2)\\
	&-\frac{h}{2}\int_{\partial_k\Omega_h}t_r^k(D_k(\gamma_1^k\beta_1^k\sigma^kA_k\sigma^k)|D_kv|^2).
	\end{split}
\end{align}
 A further use of Lemma \ref{le21} and the discrete integration by parts with respect to the difference operator $D_k$ (\ref{dk})  yield
\begin{align}\label{A8}
	&-2\int_{\Omega_h}A_k(D_k(\gamma_1^k\beta_1^k\sigma^kA_k\sigma^k)|D_kv|^2)\nonumber\\
	=&-2\int_{\Omega_h}A_kD_k(\gamma_1^k\beta_1^k\sigma^kA_k\sigma^k)A_k(|D_kv|^2)-\frac{h^2}{2}\int_{\Omega_h}D_k^2(\gamma_1^k\beta_1^k\sigma^kA_k\sigma^k)D_k(|D_kv|^2)\nonumber\\
	\begin{split}
	=&-2\int_{\Omega_h}A_kD_k(\gamma_1^k\beta_1^k\sigma^kA_k\sigma^k)|A_kD_kv|^2-\frac{h^2}{2}\int_{\Omega_h}A_kD_k(\gamma_1^k\beta_1^k\sigma^kA_k\sigma^k)|D_k^2v|^2\\
	&+\frac{h^2}{2}\int_{\Omega_{h,k}^*}D_k^3(\gamma_1^k\beta_1^k\sigma^kA_k\sigma^k)D_k|D_kv|^2-\frac{h^2}{2}\int_{\partial_k\Omega_h}D_k^2(\gamma_1^k\beta_1^k\sigma^kA_k\sigma^k)t_r^k(|D_kv|^2)n_k.
	\end{split}
\end{align}
Combining (\ref{A6})--(\ref{A8}), we have
\begin{align}\label{AIK}
\begin{split}
	I_{kk}=&\int_{\Omega_{h,k}^*}D_k(\gamma_1^k\beta_1^k\sigma^kA_k\sigma^k)|D_kv|^2-2\int_{\Omega_h}A_kD_k(\gamma_1^k\beta_1^k\sigma^kA_k\sigma^k)|A_kD_kv|^2\\ &+2\int_{\Omega_h}\gamma_1^k\beta_1^k\sigma^kD_k\sigma^k|A_kD_kv|^2-\frac{h^2}{2}\int_{\Omega_h}A_kD_k(\gamma_1^k\beta_1^k\sigma^kA_k\sigma^k)|D_k^2v|^2\\
	&+\frac{h^2}{2}\int_{\Omega_{h,k}^*}D_k^3(\gamma_1^k\beta_1^k\sigma^kA_k\sigma^k)D_k|D_kv|^2+Y_{1k},
	\end{split}
\end{align}
where
\begin{align*}
	Y_{1k}=& -\frac{h^2}{2}\int_{\partial_k\Omega_h}D_k^2(\gamma_1^k\beta_1^k\sigma^kA_k\sigma^k)t_r^k(|D_kv|^2)n_k-\frac{h}{2}\int_{\partial_k\Omega_h}t_r^k(D_k(\gamma_1^k\beta_1^k\sigma^kA_k\sigma^k)|D_kv|^2)\\
	&+\int_{\partial_k\Omega_h}\gamma_1^k\beta_1^k\sigma^kA_k\sigma^kt_r^k(|D_kv|^2)n_k.
\end{align*}
\begin{lemma}\label{leA1}
For $sh\leq \mathcal R$, 	we have
	\begin{align*}
		\gamma_1^k\beta_1^k\sigma^kD_k\sigma^k&=s\lambda\varphi O(1)+sO_{\lambda,\mathcal R}(sh),\\
		D_k(\gamma_1^k\beta_1^k\sigma^kA_k\sigma^k)&=-s\lambda^2\varphi(\sigma^k)^2(\partial_k\psi)^2+s\lambda\varphi O(1)+sO_{\lambda,\mathcal R}(sh),\\
		A_kD_k(\gamma_1^k\beta_1^k\sigma^kA_k\sigma^k)&=-s\lambda^2\varphi(\sigma^k)^2(\partial_k\psi)^2+s\lambda\varphi O(1)+sO_{\lambda,\mathcal R}(sh),\\
		h^2D_k^3(\gamma_1^k\beta_1^k\sigma^kA_k\sigma^k)&=s\lambda\varphi O(1)+sO_{\lambda,\mathcal R}(sh),
	\end{align*}
	where $\gamma_1^k=rA_k^2\rho$ and $\beta_1^k=rA_kD_k\rho$.
\end{lemma}
\begin{proof}
	By the second equality in Lemma \ref{le35} and Lemma \ref{le32}, we have
	\begin{align*}
\gamma_1^k\beta_1^k\sigma^kD_k\sigma^k&=(r\partial_k\rho+sO_{\lambda,\mathcal R}((sh)^2))\sigma^kD_k\sigma^k\\
	&=(s\lambda\varphi\partial_k\psi +sO_{\lambda,\mathcal R}((sh)^2))\sigma^kD_k\sigma^k.
	\end{align*}
Then with Assumption \ref{simg} we obtain the estimate of $\gamma_1^k\beta_1^k\sigma^kD_k\sigma^k$ since $\sigma^kD_k\sigma^k=O(1)$.

Applying Lemma \ref{le21}, we get
\begin{align*}
	D_k(\gamma_1^k\beta_1^k\sigma^kA_k\sigma^k)&=D_k(\gamma_1^k\beta_1^k)A_k(\sigma^kA_k\sigma^k)+A_k(\gamma_1^k\beta_1^k)D_k(\sigma^kA_k\sigma^k)\\
	&=D_k(\gamma_1^k\beta_1^k)((\sigma^k)^2+hO(1))+A_k(\gamma_1^k\beta_1^k)O(1),
\end{align*}
since
$\|A_k\sigma^k-\sigma^k\|_\infty\leq Ch$ and $
D_k(\sigma^kA_k\sigma^k)=D_k(\sigma^k)A^2_k\sigma^k+A_k\sigma^kA_kD_k\sigma^k=O(1)$. Moreover, by the second equality in Lemma \ref{le35} and Lemma \ref{le32}, we have
\begin{align*}
	D_k(\gamma_1^k\beta_1^k)&=\partial_k(r\partial_k\rho)+sO_{\lambda,\mathcal R}((sh)^2))=-s\lambda^2\varphi(\partial_k\psi)^2+s\lambda\varphi O(1)+sO_{\lambda,\mathcal R}((sh)^2)),\\
	A_k(\gamma_1^k\beta_1^k)&=r\partial_k\rho+sO_{\lambda,\mathcal R}((sh)^2)=s\lambda\varphi\partial_k\psi +sO_{\lambda,\mathcal R}((sh)^2).
\end{align*}
Then we get
$$D_k(\gamma_1^k\beta_1^k\sigma^kA_k\sigma^k)=-s\lambda^2\varphi(\sigma^k)^2(\partial_k\psi)^2+s\lambda\varphi O(1)+sO_{\lambda,\mathcal R}(sh).$$
Similarly, we obtain the third estimate.

To estimate $h^2D_k^3(\gamma_1^k\beta_1^k\sigma^kA_k\sigma^k)$, introducing $\alpha_1=\sigma^kA_k\sigma^k$ and $\alpha_2=\gamma_1^k\beta_1^k$ we first write
$$D_k^3(\gamma_1^k\beta_1^k\sigma^kA_k\sigma^k)=D_k^3\alpha_2A_k^3\alpha_1+3D_k^2A_k\alpha_2A_k^2D_k\alpha_1+3A_k^2D_k\alpha_2D_k^2A_k\alpha_1+A_k^3\alpha_2D_k^3\alpha_1.$$
We note that we have
$$A_k^3\alpha_1=O(1),\quad A_k^2D_k\alpha_1=O(1),\quad hD_k^2A_k\alpha_1=O(1),\quad h^2D_k^3\alpha_1=O(1),$$
and, with the second equality in  Lemma \ref{le35},
\begin{align*}
&hD_k^3\alpha_2=sO_{\lambda, \mathcal R}(1),\quad \quad D_k^2A_k\alpha_2=sO_{\lambda, \mathcal R}(1),\\& A_k^2D_k\alpha_2=sO_{\lambda, \mathcal R}(1), \quad A_k^3\alpha_2=s\lambda\varphi O(1)+sO_{\lambda, \mathcal R}((sh)^2).
\end{align*}
The estimate for $h^2D_k^3(\gamma_1^k\beta_1^k\sigma^kA_k\sigma^k)$ then follows.

\end{proof}
With  (\ref{AIK}) and Lemma \ref{leA1}, we get
\begin{align}\label{Ikk}
\begin{split}
	I_{kk}=&-\int_{\Omega_{h,k}^*}s\lambda^2\varphi(\sigma^k)^2(\partial_k\psi)^2 |D_kv|^2+\int_{\Omega_{h,k}^*}(s\lambda\varphi O(1)+sO_{\lambda\mathcal R}(sh) )|D_kv|^2
	\\
	&+2\int_{\Omega_h}s\lambda^2\varphi(\sigma^k)^2(\partial_k\psi)^2|A_kD_kv|^2+\int_{\Omega_h}(s\lambda\varphi O(1)+sO_{\lambda\mathcal R}(sh) )|A_kD_kv|^2
	\\
	&+\frac{h^2}{2}\int_{\Omega_h}s\lambda^2\varphi(\sigma^k)^2(\partial_k\psi)^2|D_k^2v|^2+h^2\int_{\Omega_h}(s\lambda\varphi O(1)+sO_{\lambda\mathcal R}(sh) )|D_k^2v|^2+Y_{1k},
	\end{split}
\end{align}
where
\begin{align*}
	Y_{1k}=& -\frac{h^2}{2}\int_{\partial_k\Omega_h}D_k^2(\gamma_1^k\beta_1^k\sigma^kA_k\sigma^k)t_r^k(|D_kv|^2)n_k-\frac{h}{2}\int_{\partial_k\Omega_h}t_r^k(D_k(\gamma_1^k\beta_1^k\sigma^kA_k\sigma^k)|D_kv|^2)\nonumber\\
	&+\int_{\partial_k\Omega_h}\gamma_1^k\beta_1^k\sigma^kA_k\sigma^kt_r^k(|D_kv|^2)n_k.
\end{align*}
	
	\subsubsection{Computation of $I_{jk}$,  $k\neq j$} The discrete integration by parts with respect to the difference operator $D_k$ (\ref{dk}) gives
	\begin{align*}
	I_{jk}=-\int_{\Omega_{h,k}^*}D_k(2\gamma_1^k\beta_1^j\sigma^jA_jD_jv)\sigma^kD_kv,
	\end{align*}
since $A_jD_jv=0$ in $\partial_k\Omega_h$ when $k\neq j$. Moreover, $I_{jk}$ can be rewritten as
\begin{align*}
	I_{jk}&=-\int_{\Omega_{h,k}^*}D_k(2\gamma_1^k\beta_1^j\sigma^j)A_kA_jD_jv\sigma^kD_kv-\int_{\Omega_{h,k}^*}A_k(2\gamma_1^k\beta_1^j\sigma^j)D_kA_jD_jv\sigma^kD_kv\\&:=I_{jk}^{(1)}+I_{jk}^{(2)},
	\end{align*}
due to Lemma \ref{le21}. Integrating by parts with respect to the average operator $A_k$ (\ref{ak}) for $I_{jk}^{(1)}$ gives
\begin{align*}
	I_{jk}^{(1)}&=-\int_{\Omega_h}A_k(D_k(2\gamma_1^k\beta_1^j\sigma^j)\sigma^kD_kv)A_jD_jv\\
	&=-\int_{\Omega_h}A_k(D_k(2\gamma_1^k\beta_1^j\sigma^j)\sigma^k)A_kD_kvA_jD_jv-\frac{h^2}{2}\int_{\Omega_h}D_k(D_k(\gamma_1^k\beta_1^j\sigma^j)\sigma^k)D^2_kvA_jD_jv,
\end{align*}
where we have used $A_jD_jv=0$ in $\partial_k\Omega_h$ for $k\neq j$ and Lemma \ref{le21}. Similarly,  we also have
\begin{align*}
I_{jk}^{(2)}	&=-\int_{\Omega_{h,jk}^*}A_j(A_k(2\gamma_1^k\beta_1^j\sigma^j)\sigma^kD_kv)D_kD_jv\\
&=-\int_{\Omega_{h,jk}^*}A_j(A_k(2\gamma_1^k\beta_1^j\sigma^j)\sigma^k)A_jD_kvD_kD_jv-\frac{h^2}{2}\int_{\Omega_{h,jk}^*}D_j(A_k(\gamma_1^k\beta_1^j\sigma^j)\sigma^k)|D_kD_jv|^2\\
&=\int_{\Omega_{h,k}^*}D_jA_j(A_k(\gamma_1^k\beta_1^j\sigma^j)\sigma^k)|D_kv|^2-\frac{h^2}{2}\int_{\Omega_{h,jk}^*}D_j(A_k(\gamma_1^k\beta_1^j\sigma^j)\sigma^k)|D_kD_jv|^2,
\end{align*}
where we have used $D_kv=0$ on $\partial_j(\Omega_{h,k}^*)$ for $k\neq j$ and $2A_jD_kvD_kD_jv=D_j(|D_kv|^2)$.

We thus have
\begin{align}\label{A11}
\begin{split}
	I_{jk}=&\int_{\Omega_{h,k}^*}D_jA_j(A_k(\gamma_1^k\beta_1^j\sigma^j)\sigma^k)|D_kv|^2-2\int_{\Omega_h}A_k(D_k(\gamma_1^k\beta_1^j\sigma^j)\sigma^k)A_kD_kvA_jD_jv\\&-\frac{h^2}{2}\int_{\Omega_h}D_k(D_k(\gamma_1^k\beta_1^j\sigma^j)\sigma^k)D^2_kvA_jD_jv-\frac{h^2}{2}\int_{\Omega_{h,jk}^*}D_j(A_k(\gamma_1^k\beta_1^j\sigma^j)\sigma^k)|D_kD_jv|^2
	\end{split}
\end{align}
\begin{lemma}\label{leA2}
For $sh\leq \mathcal R$ and $1\le  j, k\leq d$, 	we have	\begin{align*}
		D_jA_j(A_k(\gamma_1^k\beta_1^j\sigma^j)\sigma^k)&=-s\lambda^2\varphi \sigma^j\sigma^k(\partial_j\psi)^2+s\lambda\varphi O(1)+sO_{\lambda,\mathcal R}(sh),\\
		A_k(D_k(\gamma_1^k\beta_1^j\sigma^j)\sigma^k)&=-s\lambda^2\varphi \sigma^j\sigma^k(\partial_k\psi)(\partial_j\psi)+s\lambda\varphi O(1)+sO_{\lambda,\mathcal R}(sh),\\
		hD_k(D_k(\gamma_1^k\beta_1^j\sigma^j)\sigma^k)&=s\lambda\varphi O(1)+O_{\lambda,\mathcal R}(sh),\\
		D_j(A_k(\gamma_1^k\beta_1^j\sigma^j)\sigma^k)&=-s\lambda^2\varphi \sigma^j\sigma^k(\partial_j\psi)^2+s\lambda\varphi O(1)+sO_{\lambda,\mathcal R}(sh),
	\end{align*}
	where $\gamma_1^k=rA_k^2\rho$ and $\beta_1^j=rA_jD_j\rho$.
\end{lemma}
\begin{proof}
	The estimates all follow from Lemma \ref{le32} and the second equality in Lemma \ref{le35}, arguing as in the proof of Lemma \ref{leA1}.
\end{proof}

With (\ref{A11}) and Lemma \ref{leA2}, we obtain
\begin{align}\label{Ijk}
\begin{split}
	I_{jk}=&-\int_{\Omega_{h,k}^*}s\lambda^2\varphi \sigma^j\sigma^k(\partial_j\psi)^2|D_kv|^2+\int_{\Omega_{h,k}^*}(s\lambda\varphi O(1)+sO_{\lambda,\mathcal R}(sh))|D_kv|^2
	\\&+2\int_{\Omega_h}s\lambda^2\varphi \sigma^j\sigma^k(\partial_k\psi)(\partial_j\psi)A_kD_kvA_jD_jv+\int_{\Omega_h}(s\lambda\varphi O(1)+sO_{\lambda,\mathcal R}(sh))A_kD_kvA_jD_jv
	\\
	&+\frac{h^2}{2}\int_{\Omega_{h,jk}^*}s\lambda^2\varphi \sigma^j\sigma^k(\partial_j\psi)^2|D_kD_jv|^2+h^2\int_{\Omega_{h,jk}^*}(s\lambda\varphi O(1)+sO_{\lambda,\mathcal R}(sh))|D_kD_jv|^2
	\\&+h\int_{\Omega_h}(s\lambda\varphi O(1)+O_{\lambda,\mathcal R}(sh))D^2_kvA_jD_jv.
	\end{split}	
\end{align}
\subsubsection{Estimate of $(A_1v,B_1v)_{\Omega_h}$}
We now collect the terms (\ref{Ikk}) and (\ref{Ijk})  to write
$$(A_1v,B_1v)_{\Omega_h}=-\sum_{k=1}^d\int_{\Omega_{h,k}^*}s\lambda^2\varphi\sigma^k|\nabla_\sigma\psi|^2 |D_kv|^2+Y_1+\sum_{i=1}^5X_i,$$
where
\begin{align*}
	Y_1=\sum_{k=1}^dY_{1k}
	=&\sum_{k=1}^d\int_{\partial_k\Omega_h}\gamma_1^k\beta_1^k\sigma^kA_k\sigma^kt_r^k(|D_kv|^2)n_k-\frac{h}{2}\sum_{k=1}^d\int_{\partial_k\Omega_h}t_r^k(D_k(\gamma_1^k\beta_1^k\sigma^kA_k\sigma^k)|D_kv|^2)
	\\& -\frac{h^2}{2}\sum_{k=1}^d\int_{\partial_k\Omega_h}D_k^2(\gamma_1^k\beta_1^k\sigma^kA_k\sigma^k)t_r^k(|D_kv|^2)n_k,
\end{align*}

\begin{align*}
	X_{1}&=2\sum_{k=1}^d\int_{\Omega _h}s\lambda^2\varphi(\sigma^k)^2(\partial_k\psi)^2|A_kD_kv|^2+2\sum_{j,k=1 \atop j\neq k}^d\int_{\Omega_h}s\lambda^2\varphi \sigma^j\sigma^k(\partial_k\psi)(\partial_j\psi)A_kD_kvA_jD_jv\\
	&=2\int_{\Omega_h}s\lambda^2\varphi\left|\sum_{k=1}^d\sigma^k(\partial_k\psi)A_kD_kv\right|^2\geq 0,
\end{align*}

$$X_2=\frac{h^2}{2}\sum_{k=1}^d\int_{\Omega_h}s\lambda^2\varphi(\sigma^k)^2(\partial_k\psi)^2|D_k^2v|^2+\frac{h^2}{2}\sum_{j,k=1 \atop j\neq k}^d\int_{\Omega_{h,jk}^*}s\lambda^2\varphi \sigma^j\sigma^k(\partial_j\psi)^2|D_kD_jv|^2\geq 0,$$

$$X_3=\sum_{k=1}^d\int_{\Omega_{h,k}^*}(s\lambda\varphi O(1)+sO_{\lambda,\mathcal R}(sh) )|D_kv|^2+\sum_{k=1}^d\int_{\Omega_h}(s\lambda\varphi O(1)+sO_{\lambda,\mathcal R}(sh) )|A_kD_kv|^2,$$

$$X_4=\sum_{j,k=1 \atop j\neq k}^d\int_{\Omega_h}(s\lambda\varphi O(1)+sO_{\lambda,\mathcal R}(sh))A_kD_kvA_jD_jv$$
and
\begin{align*}
X_5=&\sum_{j,k=1 \atop j\neq k}^d\int_{\Omega_{h,jk}^*}h^2(s\lambda\varphi O(1)+sO_{\lambda,\mathcal R}(sh))|D_kD_jv|^2+\sum_{k=1}^d\int_{\Omega_h}h^2(s\lambda\varphi O(1)+sO_{\lambda,\mathcal R}(sh) )|D_k^2v|^2\\
&+\sum_{j,k=1 \atop j\neq k}^d\int_{\Omega_h}h(s\lambda\varphi O(1)+O_{\lambda,\mathcal R}(sh))D^2_kvA_jD_jv.
\end{align*}

We conclude with Cauchy-Schwarz inequalities that yields
$$|X_4|\leq\sum_{k=1}^d\int_{\Omega_h}(s\lambda\varphi |O(1)|+s|O_{\lambda,\mathcal R}(sh)|)|A_kD_kv|^2,$$
and
\begin{align*}
	|X_5|
\leq &\sum_{k=1}^d\int_{\Omega_h}h^2(s\lambda\varphi |O(1)|+s|O_{\lambda,\mathcal R}(sh)|)|D_k^2v|^2+\sum_{k=1}^d\int_{\Omega_h}(s\lambda\varphi |O(1)|+|O_{\lambda,\mathcal R}(sh)|)|A_kD_kv|^2\\&+\sum_{j,k=1 \atop j\neq k}^d\int_{\Omega_{h,jk}^*}h^2(s\lambda\varphi |O(1)|+s|O_{\lambda,\mathcal R}(sh)|)|D_kD_jv|^2.
\end{align*}
Then we complete the proof of Lemma \ref{le37}.

\subsection{Proof of Lemma \ref{le39}}

From the definitions of $A_2v$ and $B_1v$  in (\ref{abg}) we have
\begin{align*}
	(A_2v,B_1v)_{\Omega_h}&=\sum_{j,k=1}^d\int_{\Omega_h}2\sigma^jr^2D_j^2\rho A_j^2v\sigma^kA_kD_k\rho A_kD_kv\nonumber\\
	&:=\sum_{j,k=1}^d\int_{\Omega_h}2\gamma_2^j \beta_1^k\sigma^j\sigma^kA_j^2vA_kD_kv\nonumber\\
	&:=\sum_{j,k=1}^dJ_{jk},
\end{align*}
where $\gamma_2^j=rD_j^2\rho$ and $\beta_1^k=rA_kD_k\rho$.
\subsubsection{Computation of $J_{kk}$}
Since $2A_k^2vD_kA_kv=D_k(|A_kv|^2)$, we have
\begin{align*}
	J_{kk}=&\int_{\Omega_h}\gamma_2^k \beta_1^k(\sigma^k)^2D_k(|A_kv|^2)\\
	=&-\int_{\Omega_{h,k}^*}D_k(\gamma_2^k \beta_1^k(\sigma^k)^2)|A_kv|^2+\int_{\partial_k\Omega_h}\gamma_2^k \beta_1^k(\sigma^k)^2t_r^k(|A_kv|^2)n_k\\
	=&-\int_{\Omega_{h,k}^*}D_k(\gamma_2^k \beta_1^k(\sigma^k)^2)A_k(|v|^2)+\frac{h^2}{4}\int_{\Omega_{h,k}^*}D_k(\gamma_2^k \beta_1^k(\sigma^k)^2)|D_kv|^2\\
	&+\int_{\partial_k\Omega_h}\gamma_2^k \beta_1^k(\sigma^k)^2t_r^k(|A_kv|^2)n_k\\
	=&-\int_{\Omega_h}A_kD_k(\gamma_2^k \beta_1^k(\sigma^k)^2)|v|^2+\frac{h^2}{4}\int_{\Omega_{h,k}^*}D_k(\gamma_2^k \beta_1^k(\sigma^k)^2)|D_kv|^2\\
	&+\int_{\partial_k\Omega_h}\gamma_2^k \beta_1^k(\sigma^k)^2t_r^k(|A_kv|^2)n_k,
\end{align*}
where we have used (\ref{dk}), (\ref{ak}) and the fact that $u=0$ on $\partial_k\Omega_h$.
\begin{lemma}
	For $sh\leq \mathcal R$, 	we have
	\begin{align*}
		A_kD_k(\gamma_2^k \beta_1^k(\sigma^k)^2)&=-3s^3\lambda^4\varphi^3(\sigma^k)^2(\partial_k\psi)^4+(s\lambda\varphi)^3O(1)+s^2O_{\lambda,\mathcal R}(1)+s^3O_{\lambda,\mathcal R}((sh)^2),\\
		D_k(\gamma_2^k \beta_1^k(\sigma^k)^2)&=s^3O_{\lambda,\mathcal R}(1),
	\end{align*}
	where $\gamma_2^k=rD_k^2\rho$ and $\beta_1^k=rA_kD_k\rho$.
\end{lemma}
\begin{proof}
	The estimates follow from the first equality in Lemma \ref{le35} and Lemma \ref{le33} arguing as in the proof of Lemma \ref{leA1}.
\end{proof}
Then we get
\begin{align}\label{Jkk}
\begin{split}
	J_{kk}=&3\int_{\Omega_h}s^3\lambda^4\varphi^3(\sigma^k)^2(\partial_k\psi)^4|v|^2+\int_{\Omega_h}((s\lambda\varphi)^3O(1)+s^2O_{\lambda,\mathcal R}(1)+s^3O_{\lambda,\mathcal R}((sh)^2))|v|^2\\
	&+\int_{\Omega_{h,k}^*}sO_{\lambda,\mathcal R}((sh)^2))|D_kv|^2+Y_{2k},
	\end{split}
\end{align}
where $$Y_{2k}=\int_{\partial_k\Omega_h}\gamma_2^k \beta_1^k(\sigma^k)^2t_r^k(|A_kv|^2)n_k.$$
\subsubsection{Computation of $J_{jk}$, $j\neq k$}
The discrete integration by parts with respect to  the average operator $A_j$ (\ref{ak}) gives
$$J_{jk}=\int_{\Omega_{h,j}^*}A_j(2\gamma_2^j \beta_1^k\sigma^j\sigma^kA_kD_kv)A_jv,$$
since $A_kD_kv=0$ on $\partial_j\Omega_h$ when $k\neq j$. Moreover, $J_{jk}$ can be rewritten as
\begin{align}\label{A15}
	J_{jk}&=\int_{\Omega_{h,j}^*}A_j(2\gamma_2^j \beta_1^k\sigma^j\sigma^k)A_jA_kD_kvA_jv+\frac{h^2}{2}\int_{\Omega_{h,j}^*}D_j(\gamma_2^j \beta_1^k\sigma^j\sigma^k)D_jA_kD_kvA_jv\nonumber\\
	&:=J_{jk}^{(1)}+\frac{h^2}{2}\int_{\Omega_{h,j}^*}D_j(\gamma_2^j \beta_1^k\sigma^j\sigma^k)D_jA_kD_kvA_jv.
\end{align}
due to Lemma \ref{le21}.  Further, by (\ref{ak}) and  Lemma \ref{le21}, we have
\begin{align}
	J_{jk}^{(1)}&=\int_{\Omega_{h,jk}^*}A_k(A_j(2\gamma_2^j \beta_1^k\sigma^j\sigma^k)A_jv)A_jD_kv\nonumber\\
	&=\int_{\Omega_{h,jk}^*}A_kA_j(2\gamma_2^j \beta_1^k\sigma^j\sigma^k)A_kA_jvA_jD_kv+\frac{h^2}{2}\int_{\Omega_{h,jk}^*}D_kA_j(\gamma_2^j \beta_1^k\sigma^j\sigma^k)|A_jD_kv|^2\nonumber\\
	&=\int_{\Omega_{h,jk}^*}A_kA_j(\gamma_2^j \beta_1^k\sigma^j\sigma^k)D_k(|A_jv|^2)+\frac{h^2}{2}\int_{\Omega_{h,jk}^*}D_kA_j(\gamma_2^j \beta_1^k\sigma^j\sigma^k)|A_jD_kv|^2\nonumber\\
	&:=J_{jk}^{(11)}+\frac{h^2}{2}\int_{\Omega_{h,jk}^*}D_kA_j(\gamma_2^j \beta_1^k\sigma^j\sigma^k)|A_jD_kv|^2,
\end{align}
where we have used $A_jv=0$ on $\partial_k(\Omega_{h,j}^*)$ for $k\neq j$.  Using (\ref{dk}), Lemma \ref{le21} and (\ref{ak}), we obtain
\begin{align}\label{A17}
	J_{jk}^{(11)}&=-\int_{\Omega_{h,j}^*}D_kA_kA_j(\gamma_2^j \beta_1^k\sigma^j\sigma^k)|A_jv|^2\nonumber\\
	&=-\int_{\Omega_{h,j}^*}D_kA_kA_j(\gamma_2^j \beta_1^k\sigma^j\sigma^k)A_j(|v|^2)+\frac{h^2}{4}\int_{\Omega_{h,j}^*}D_kA_kA_j(\gamma_2^j \beta_1^k\sigma^j\sigma^k)|D_jv|^2\nonumber\\
	&=-\int_{\Omega_h}D_kA_kA_j^2(\gamma_2^j \beta_1^k\sigma^j\sigma^k)|v|^2+\frac{h^2}{4}\int_{\Omega_{h,j}^*}D_kA_kA_j(\gamma_2^j \beta_1^k\sigma^j\sigma^k)|D_jv|^2,
\end{align}
by virtue of  $A_jv=0$ on $\partial_k(\Omega_{h,j}^*)$ for $k\neq j$ and $v=0$ on $\partial_j\Omega_h$.

Combing (\ref{A15})-(\ref{A17}), we have
\begin{align}\label{A18}
	\begin{split}
		J_{jk}=&-\int_{\Omega_h}D_kA_kA_j^2(\gamma_2^j \beta_1^k\sigma^j\sigma^k)|v|^2+\frac{h^2}{2}\int_{\Omega_{h,jk}^*}D_kA_j(\gamma_2^j \beta_1^k\sigma^j\sigma^k)|A_jD_kv|^2\\
		&+\frac{h^2}{4}\int_{\Omega_{h,j}^*}D_kA_kA_j(\gamma_2^j \beta_1^k\sigma^j\sigma^k)|D_jv|^2+\frac{h^2}{2}\int_{\Omega_{h,j}^*}D_j(\gamma_2^j \beta_1^k\sigma^j\sigma^k)D_jA_kD_kvA_jv.
	\end{split}
\end{align}
\begin{lemma}\label{leA4}
	For $sh\leq \mathcal R$ and $1\le  j, k\leq d$, 	we have
	\begin{align*}
		D_kA_kA_j^2(\gamma_2^j \beta_1^k\sigma^j\sigma^k)=&-3s^3\lambda^4\varphi^3\sigma_j\sigma^k(\partial_k\psi)^2(\partial_j\psi)^2\\
		&+(s\lambda\varphi)^3O(1)+s^2O_{\lambda,\mathcal R}(1)+s^3O_{\lambda,\mathcal R}((sh)^2),\\
		D_kA_j(\gamma_2^j \beta_1^k\sigma^j\sigma^k)=&s^3O_{\lambda,\mathcal R}(1),\\
		 D_kA_kA_j(\gamma_2^j \beta_1^k\sigma^j\sigma^k)=&s^3O_{\lambda,\mathcal R}(1),\\
		\quad D_j(\gamma_2^j \beta_1^k\sigma^j\sigma^k)=&s^3O_{\lambda,\mathcal R}(1),
	\end{align*}
	where $\gamma_2^j=rD_j^2\rho$ and $\beta_1^k=rA_kD_k\rho$.
\end{lemma}
\begin{proof}
	The estimates follow from the first equality in Lemma \ref{le35} and Lemma \ref{le33} arguing as in the proof of Lemma \ref{leA1}.
\end{proof}
Using Lemma \ref{leA4} and the facts that $|A_jD_kv|^2\leq A_j(|D_kv|^2)$ and $D_kv=0$ on $\partial_j(\Omega_{h,k}^*)$, we get

\begin{align}
	\frac{h^2}{2}\left|\int_{\Omega_{h,jk}^*}D_kA_j(\gamma_2^j \beta_1^k\sigma^j\sigma^k)|A_jD_kv|^2\right|&\leq \int_{\Omega_{h,jk}^*}s|O_{\lambda,\mathcal R}((sh)^2)|A_j(|D_kv|^2)\nonumber\\
	&=\int_{\Omega_{h,k}^*}s|O_{\lambda,\mathcal R}((sh)^2)||D_kv|^2.
\end{align}
By Young's inequality and Lemma \ref{leA4} we note that
\begin{align}\label{A20}
	&\frac{h^2}{2}\left|\int_{\Omega_{h,j}^*}D_j(\gamma_2^j \beta_1^k\sigma^j\sigma^k)D_jA_kD_kvA_jv\right|\nonumber\\
	\leq &s^3(sh)\int_{\Omega_{h,j}^*}|O_{\lambda,\mathcal R}(1)||A_jv|^2+sh^2(sh)\int_{\Omega_{h,j}^*}|O_{\lambda,\mathcal R}(1)||D_jD_kA_kv|^2\nonumber\\
	\leq &s^3(sh)\int_{\Omega_{h,j}^*}|O_{\lambda,\mathcal R}(1)|A_j(|v|^2)+s(sh)\int_{\Omega_{h,j}^*}|O_{\lambda,\mathcal R}(1)|A_j(|D_kA_kv|^2)\nonumber\\
	=&s^3\int_{\Omega_h}|O_{\lambda,\mathcal R}(sh)||v|^2+s\int_{\Omega_h}|O_{\lambda,\mathcal R}(sh)|D_kA_kv|^2,
\end{align}
since $v=0$ on $\partial_j\Omega_h$ and $D_kA_kv=0$ on $\partial_j\Omega_h$ for $k\neq j$ and using Corollary \ref{co22}.

In terms of  (\ref{A18})-(\ref{A20}) and Lemma \ref{leA4}, we obtain
\begin{align}\label{Jjk}
	\begin{split}
		J_{jk}\geq &3\int_{\Omega_h}s^3\lambda^4\varphi^3\sigma_j\sigma^k(\partial_k\psi)^2(\partial_j\psi)^2|v|^2-\int_{\Omega_h}s|O_{\lambda,\mathcal R}(sh)||D_kA_kv|^2\\
		&-\int_{\Omega_h}((s\lambda\varphi)^3|O(1)|+s^2|O_{\lambda,\mathcal R}(1)|+s^3|O_{\lambda,\mathcal R}(sh)||v|^2
		\\
		&-\int_{\Omega_{h,k}^*}s|O_{\lambda,\mathcal R}((sh)^2)||D_kv|^2-\int_{\Omega_{h,j}^*}s|O_{\lambda,\mathcal R}((sh)^2)||D_jv|^2	\end{split}
\end{align}
\subsubsection{Estimate of $(A_2v,B_1v)_{\Omega_h}$}
Combining (\ref{Jkk}) and (\ref{Jjk}),  the Lemma follows.
\subsection{Proof of Lemma \ref{le310}}

From the definitions of $A_1v$ and $B_2v$  in (\ref{abg}) we have
\begin{align}\label{bc1}
	(A_1v,B_2v)_{\Omega_h}=-\sum_{k=1}^d\int_{\Omega_h} 2srA_k^2\rho D_k(\sigma^kD_kv)(\Delta_\sigma\varphi) v
	\end{align}
For concision we now set $\omega_1^k=rA_k^2\rho(\Delta_\sigma\varphi)$. The discrete integration by parts (\ref{dk}) yields
\begin{align}\label{bc2}
	-\int_{\Omega_h} 2s\omega_1^k D_k(\sigma^kD_kv) v&=\int_{\Omega_{h,k}^*}2s D_k(\omega_1^kv) \sigma^kD_kv\nonumber\\
	&=\int_{\Omega_{h,k}^*} 2sA_k(\omega_1^k)\sigma^k|D_kv|^2+\int_{\Omega_{h,k}^*}2s D_k(\omega_1^k)\sigma^kA_kvD_kv,
\end{align}
where we have used $v=0$ on $\partial_k\Omega_h$ and Lemma \ref{le21}.

\begin{lemma}\label{leA5}
For $sh\leq \mathcal R$, 	we have	\begin{align*}
		A_k(\omega_1^k)&=\lambda^2\varphi|\nabla_\sigma\psi|^2+\lambda\varphi O(1)+O_{\lambda,\mathcal R}(h+(sh)^2) ,\\
		D_k(\omega_1^k)&=O_{\lambda,\mathcal R}(1),
	\end{align*}
	where $\omega_1^k=rA_k^2\rho(\Delta_\sigma\varphi)$.
\end{lemma}
\begin{proof}Observing that
	$$\Delta_\sigma\varphi=\sum_{k=1}^d\sigma^k\partial_k^2\varphi=-\sum_{k=1}^d\sigma^k\partial_k(\lambda\varphi\partial_k\psi)=\lambda^2\varphi|\nabla_\sigma\psi|^2+\lambda\varphi O(1),$$
	by the first equality in Lemma \ref{le34} and the second equality in Lemma \ref{lexx}, we obtain the desired results.
\end{proof}
By  the discrete integration by parts (\ref{ak}) and Young's inequality, we obtain
\begin{align}\label{bc3}
	\left|\int_{\Omega_{h,k}^*}s O_{\lambda,\mathcal R}(1)\sigma^kA_kvD_kv\right|&\leq \int_{\Omega_{h,k}^*}s^2 |O_{\lambda,\mathcal R}(1)||A_kv |^2+\int_{\Omega_{h,k}^*} |O_{\lambda,\mathcal R}(1)||D_kv |^2\nonumber\\
	&\leq \int_{\Omega_{h,k}^*}s^2 |O_{\lambda,\mathcal R}(1)|A_k(|v |^2)+\int_{\Omega_{h,k}^*} |O_{\lambda,\mathcal R}(1)||D_kv |^2\nonumber\\
	&=\int_{\Omega_h}s^2 |O_{\lambda,\mathcal R}(1)||v |^2+\int_{\Omega_k^*} |O_{\lambda,\mathcal R}(1)||D_kv |^2,
\end{align}
since $|A_kv|^2\leq A_k(|v|^2)$ and $v=0$ on $\partial_k\Omega_h$.

Consequently, we obtain the estimate of $(A_1v,B_2v)_{\Omega_h}$  from (\ref{bc1}), (\ref{bc2}), Lemma \ref{A5} and (\ref{bc3}).

Thus the proof of  Lemma \ref{le310} is complete.

\subsection{Proof of Lemma \ref{le311}}

From the definitions of $A_2v$ and $B_2v$  in (\ref{abg}) we have
\begin{align}\label{bc4}
	(A_2v,B_2v)_{\Omega_h} &=-\sum_{k=1}^d\int_{\Omega_h} 2s\sigma^krD_k^2\rho A_k^2v(\Delta_\sigma\varphi) v\nonumber\\
	&=\sum_{k=1}^d\left(-\int_{\Omega_h}2s\sigma^krD_k^2\rho (\Delta_\sigma\varphi) |v|^2-\frac{h^2}{2}\int_{\Omega_h} s\sigma^krD_k^2\rho (\Delta_\sigma\varphi)vD_k^2v\right)\nonumber\\
	&:=\sum_{k=1}^d\left(-\int_{\Omega_h} 2s\omega_2^k |v|^2-\frac{h^2}{2}\int_{\Omega_h} s\omega_2^kvD_k^2v\right)\nonumber\\
	&:=Q_{1k}+Q_{2k},
	\end{align}
where $\omega_2^k=\sigma^k rD_k^2\rho(\Delta_\sigma\varphi)$ and we have used $A_k^2v=v+\frac{h^2}{4}D_k^2v$.
 By using integration by parts (\ref{dk}) twice, we can prove
\begin{align}\label{bc5}
	Q_{2k}&=\frac{sh^2}{2}\int_{\Omega_{h,k}^*}D_k( \omega_2^kv)D_kv\nonumber\\
	&=\frac{sh^2}{2}\int_{\Omega_{h,k}^*}A_k( \omega_2^k)|D_kv|^2+\frac{sh^2}{4}\int_{\Omega_{h,k}^*}D_k( \omega_2^k)D_k(|v|^2)\nonumber\\
	&=\frac{sh^2}{2}\int_{\Omega_{h,k}^*}A_k( \omega_2^k)|D_kv|^2-\frac{sh^2}{4}\int_{\Omega_h}D_k^2( \omega_2^k)|v|^2,
\end{align}
where we have used $v=0$ on $\partial_k\Omega_h$ and $A_kvD_kv=\frac{1}{2}D_k(|v|^2)$.

Similar to Lemma \ref{leA5} and Lemma \ref{leA1}, we have the following Lemma.
\begin{lemma}\label{bc6}
For $sh\leq \mathcal R$, 	we have
\begin{align*}
	\omega_2^k&=s^2\lambda^4\varphi^3|\nabla_\sigma\psi|^2\sigma^k(\partial_k\psi)^2+s^2\lambda^3\varphi^3O(1)+sO_\lambda(1)+s^2O_{\lambda,\mathcal R}((sh)^2),\\
	A_k( \omega_2^k)&=s^2O_{\lambda,\mathcal R}(1),\\
	h^2D_k^2( \omega_2^k)&=sO_{\lambda,\mathcal R}(sh),
	\end{align*}
	where $\omega_2^k=\sigma^k rD_k^2\rho(\Delta_\sigma\varphi)$.
\end{lemma}
Then we get the estimate of $(A_2v,B_2v)_{\Omega_h}$ from (\ref{bc4}), (\ref{bc5}) and Lemma \ref{bc6}.
Thus the proof of Lemma \ref{le311} is complete.

\subsection{Proof of Lemma \ref{le313}}

We set $\omega_3^k=\varphi\sigma^k|\nabla_\sigma\psi|^2$.
For $k,j=1,\dots,d$ with $k\neq j$, by (\ref{ak}), (\ref{dk}) and Lemma \ref{le21}, we have
\begin{align*}
	\int_{\Omega_{h,k}^*}\omega_3^k|D_kv|^2=&\int_{\Omega_{h,kj}^*}A_j(\omega_3^k|D_kv|^2)\\
	=&\int_{\Omega_{h,kj}^*}A_j(\omega_3^k)A_j(|D_kv|^2)+\frac{h^2}{4}\int_{\Omega_{h,kj}^*}D_j(\omega_3^k)D_j(|D_kv|^2)\\
	=&\int_{\Omega_{h,kj}^*}A_j(\omega_3^k)|A_jD_kv|^2+\frac{h^2}{4}\int_{\Omega_{h,kj}^*}A_j(\omega_3^k)|D_jD_kv|^2-\frac{h^2}{4}\int_{\Omega_{h,k}^*}D_j^2(\omega_3^k)|D_kv|^2\\
	\geq &\frac{h^2}{4}\int_{\Omega_{h,kj}^*}A_j(\omega_3^k)|D_jD_kv|^2-\frac{h}{4}\int_{\Omega_{h,k}^*}((\tau_j-\tau_{-j}) D_j(\omega_3^k))|D_kv|^2,
	\end{align*}
since $D_kv=0$ on $\partial_j(\Omega_{h,k}^*)$.

On the other hand, for $k=1,\dots,d$,
by (\ref{ak}), (\ref{dk}) and Lemma \ref{le21}, we have
\begin{align*}
	\int_{\Omega_{h,k}^*}\omega_3^k|D_kv|^2=&\int_{\Omega_h}A_k(\omega_3^k|D_kv|^2)+\frac{h}{2}\int_{\partial_k\Omega_h}t_r^k(\omega_3^k|D_kv|^2)\\
	\geq &\int_{\Omega_h}A_k(\omega_3^k)A_k(|D_kv|^2)+\frac{h^2}{4}\int_{\Omega_h}D_k(\omega_3^k)D_k(|D_kv|^2)\\
	=&\int_{\Omega_h}A_k(\omega_3^k)|A_kD_kv|^2+\frac{h^2}{4}\int_{\Omega_h}A_k(\omega_3^k)|D_k^2v|^2\\
	&-\frac{h}{4}\int_{\Omega_{h,k}^*}((\tau_k-\tau_{-k})D_k(\omega_3^k))|D_kv|^2+\frac{h^2}{4}\int_{\partial_k\Omega_h}D_k(\omega_3^k)t_r^k(|D_kv|^2)n_k.
		\end{align*}

We thus have
\begin{align}\label{bc7}
\begin{split}
\sum_{j,k=1}^d	\int_{\Omega_{h,k}^*}\omega_3^k|D_kv|^2\geq &\sum_{j,k=1 \atop j\neq k}^d\left(\frac{h^2}{4}\int_{\Omega_{h,kj}^*}A_j(\omega_3^k)|D_jD_kv|^2-\frac{h}{4}\int_{\Omega_{h,k}^*}((\tau_j-\tau_{-j}) D_j(\omega_3^k))|D_kv|^2\right)\\
&+\sum_{k=1}^d\left(\frac{h^2}{4}\int_{\Omega_h}A_k(\omega_3^k)|D_k^2v|^2-\frac{h}{4}\int_{\Omega_{h,k}^*}((\tau_k-\tau_{-k})D_k(\omega_3^k))|D_kv|^2\right)\\
&+\sum_{k=1}^d\int_{\Omega_h}A_k(\omega_3^k)|A_kD_kv|^2+Y_3,
\end{split}
\end{align}	
where $$Y_3=	\frac{h^2}{4}\sum_{k=1}^d\int_{\partial_k\Omega_h}D_k(\omega_3^k)t_r^k(|D_kv|^2)n_k.$$
Moreover, for $j,k=1,\dots,d$, we have
\begin{align}\label{bc8}
	A_j(\omega_3^k)=\omega_3^k+hO_{\lambda, \mathcal R}(1), \quad
	(\tau_j-\tau_{-j}) D_k(\omega_3^k)=O_{\lambda, \mathcal R}(1),
\end{align}
	where $\omega_3^k=\varphi\sigma^k|\nabla_\sigma\psi|^2$.
 From (\ref{bc7}) and (\ref{bc8}),	the result of Lemma \ref{le313} follows.

\subsection{Proof of Lemma \ref{le314}}
We begin proving the first inequality (\ref{e310}) of our Lemma. Recalling that $v=ru, u=\rho v$, thanks to Lemma \ref{le21} and Young's inequality, we have
\begin{align*}
	\int_{\Omega_{h,k}^*}r^2|D_ku|^2&\leq 2\left(\int_{\Omega_{h,k}^*}r^2|D_kvA_k\rho|^2+\int_{\Omega_{h,k}^*}r^2|A_kvD_k\rho|^2\right)\\
	&:=2(H_1+H_2).
\end{align*}
Let us first estimate $H_2$. Using Lemma \ref{le21}, (\ref{ak}) and the second equality in Lemma \ref{le34}  we botain
\begin{align*}
	H_2\leq \int_{\Omega_{h,k}^*}|rD_k\rho|^2A_k(|v|^2)=\int_{\Omega_h}A_k(|rD_k\rho|^2)|v|^2\leq s^2\int_{\Omega_h}|O_{\lambda,\mathcal R}(1)| |v|^2,
\end{align*}
since $v=0$ on $\partial_k\Omega_h$.
It remains to prove that
$$H_1\leq \int_{\Omega_{h,k}^*}|O_{\lambda,\mathcal R}(1)| |D_kv|^2,$$
which follows from the first equality in Lemma \ref{le34}, and inequality (\ref{e310}) is proved.

To prove the inequality (\ref{e311}), we note that
$$|\rho D_kv|^2=|\rho A_krD_ku+\rho D_krA_ku|^2\leq C_{\lambda,\mathcal R}(|D_ku|^2+s^2|A_ku|^2),$$
due to Lemma \ref{le21} and the first and second equality Lemma \ref{le34}. Then, by virtue of  $t_r^k(|A_ku|^2)=\frac{h^2}{4}t_r^k(|D_ku|^2)$ and $sh\leq \mathcal R$, we have
\begin{align}\label{bc9}
	t_r^k(|D_kv|^2)=t_r^k(r^2|\rho D_kv|^2)\leq C_{\lambda,\mathcal R}t_r^k(r^2)t_r^k(|D_ku|^2)
\end{align}
Moreover, by Definition \ref{deno}, we have
\begin{align}\label{bc10}
|\partial_nu|^2=|t_r^k(\sigma^kD_ku)n_k|^2=|t_r^k(\sigma^k)|^2|t_r^k(D_ku)|^2\geq Ct_r^k(|D_ku|^2),	
\end{align}
for  $x\in\partial_k\Omega_h$, $k=1,\dots,d$.
 Combining (\ref{bc9}) with (\ref{bc10}), we get (\ref{e311}).
Thus the proof of Lemma \ref{le314} is complete.

 \section{}
\begin{proposition}\label{regu}
	Let  Assumption \ref{assq} and  Assumption \ref{as2} hold. Let $M>0$,  $ q\in C(\Omega_h)$ satisfying $\|q\|_{L_h^\infty(\Omega_h)}\leq M$  and let $u\in\mathcal C(\overline\Omega_h)$ satisfying
\begin{align}\label{BB}
\begin{split}
		-\Delta_hu+qu&=f,\quad {\rm in}~\Omega_h,\\
		u&=0,\quad {\rm on}~\partial\Omega_h,
		\end{split}
	\end{align}
	where $f\in\mathcal C(\Omega_h)$ and $f=0$ on $\partial\Omega_h$. Then
 there exists constant $C>0$ independent of $h$ such that
\begin{align}\label{B1}
\|u\|_{H_h^2(\Omega_h)}\leq C\|f\|_{L_h^2(\Omega_h)}.	
\end{align}

\end{proposition}
\begin{proof}
By Assumption \ref{as2},	 we have
\begin{align}\label{B2}
\|u\|_{H_h^1(\Omega_h)}\leq C\|f\|_{L_h^2(\Omega_h)},
\end{align}	
for some constant $C>0$ independent of $h$. Accordingly, replacing $f$ by $f-qu$, we are reduced to the case $q=0$, that we assume from now.

Since $\Omega_h=h\mathbb Z^d\cap (0,1)^d$, we first propose to extend $u$ a priori defined on the 	discrete domain $\Omega_h$ to $\Omega_{{\rm ext},h}=h\mathbb Z^d\cap (-1,2)^d$ as follows.  For $x=(x_1,\dots, x_d)\in\left([0,1]^{d-1}\times (-1,2)\right)\cap\Omega_{{\rm ext},h}$, we set $\widetilde u(x)=-u(x_1,\dots,x_{d-1},-x_d)$ for $x_d\in(-1,0)$ and 	$\widetilde u(x)=-u(x_1,\dots,x_{d-1},1-(x_d-1))$ for $x_d\in(1,2)$. This defines $\widetilde u$ on $\left([0,1]^{d-1}\times (-1,2)\right)\cap\Omega_{{\rm ext},h}$. Then, for $x=(x_1,\dots, x_d)\in\left([0,1]^{d-2}\times (-1,2)\times[0,1]\right)\cap\Omega_{{\rm ext},h}$, we set $\widetilde u(x)=-u(x_1,\dots,x_{d-2},-x_{d-1},x_d)$ for $x_{d-1}\in(-1,0)$ and 	$\widetilde u(x)=-u(x_1,\dots,1-(x_{d-1}-1),x_d)$ for $x_{d-1}\in(1,2)$. This defines $\widetilde u$ on $\left([0,1]^{d-2}\times (-1,2)\times[0,1]\right)\cap\Omega_{{\rm ext},h}.$ By induction, we then extend it for $x=(x_1,\dots, x_d)\in\Omega_{{\rm ext},h}$ by setting $\widetilde u=-u(-x_1,\dots, x_d)$ for $x_1\in(-1,0)$ and $\widetilde u=-u(1-(x_1-1),\dots, x_d)$ for $x_1\in(1,2)$. We do a similar extension $\widetilde f$ of $f$ on $\Omega_{{\rm ext},h}$ taking care of choosing $\widetilde f=0$ on $\partial\Omega_h$.

We thus have constructed a solution $\widetilde u$ of
\begin{align}\label{B4}
\begin{split}
		-\Delta_h\widetilde u&=\widetilde f,\quad {\rm in}~\Omega_{{\rm ext},h},\\
	\widetilde	u&=0,\quad {\rm on}~\partial\Omega_{{\rm ext},h}.
		\end{split}
	\end{align}	
We then choose a function $\chi\in C_c^\infty((-1,2)^d)$ such that $\chi=1$ on $[0,1]^d$ and we multiply (\ref{B4}) by $-\chi D_1^2\widetilde u$, after some integrations by parts where all the boundary terms vanish due to the choice of $\chi$, we obtain
\begin{align}\label{B5}
	\begin{split}
		&\int_{\Omega_{ext,h}}\chi A_1\sigma^1|D_1^2\widetilde u|^2+\sum_{k=2}^d\int_{(\Omega_{ext,h})_1^*}A_1(A_k\chi\sigma^k)|D_1D_k\widetilde u|^2\\
		=&-\int_{\Omega_{ext,h}}\chi D_1^2\widetilde u\widetilde f-\int_{\Omega_{ext,h}}\chi D_1\sigma^1D_1^2\widetilde uA_1D_1\widetilde u
		+\sum_{k=2}^d\int_{\Omega_{ext,h}}A_k(D_k\chi\sigma^kD_k\widetilde u)D_1^2\widetilde u\\
		&-\sum_{k=2}^d\int_{(\Omega_{ext,h})_1^*}D_1(A_k\chi\sigma^k)A_1D_k\widetilde u D_kD_1\widetilde u.
			\end{split}
\end{align}	
	Of course, since $\chi=1$ on $[0,1]^d$, the left hand-side of (\ref{B5}) is bounded from below by
	$$C\left(\int_{\Omega_h}|D_1^2u|^2+\sum_{k=2}^d\int_{\Omega_{h,k1}^*}|D_1D_ku|^2\right).$$
	On the other hand, noting that $\widetilde u$ and $\widetilde f$ are symmetric extensions of $u$ and $f$, the right hand-side of (\ref{B5}) is bounded form above by
	$$C\left(\left(\int_{\Omega_h}|D_1^2u|^2\right)^\frac{1}{2}
	+\sum_{k=2}^d\left(\int_{\Omega_{h,k1}^*}|D_1D_ku|^2\right)^\frac{1}{2}\right)\left(\|f\|^2_{L^2_h(\Omega_h)}+\|u\|_{H^1_h(\Omega_h)}\right).$$
We thus obtain
\begin{align}\label{B6}
	\left(\int_{\Omega_h}|D_1^2u|^2\right)^\frac{1}{2}
	+\sum_{k=2}^d\left(\int_{\Omega_{h,k1}^*}|D_1D_ku|^2\right)^\frac{1}{2}\leq C\left(\|f\|^2_{L^2_h(\Omega_h)}+\|u\|_{H^1_h(\Omega_h)}\right).
\end{align}	
Similarly, we have 	
\begin{align}\label{B7}
	\left(\int_{\Omega_h}|D_i^2u|^2\right)^\frac{1}{2}
	+\sum_{k=2}^d\left(\int_{\Omega_{h,ki}^*}|D_iD_ku|^2\right)^\frac{1}{2}\leq C\left(\|f\|^2_{L^2_h(\Omega_h)}+\|u\|_{H^1_h(\Omega_h)}\right),
\end{align}		
	for $i=2,\dots,d$.
	Combining (\ref{B6}), (\ref{B7}) and (\ref{B2}), we obtain the inequality (\ref{B1}).	
	
\end{proof}

\end{appendix}

\section*{Acknowledgements}

The research of XZ was supported in part by National Key R\&D Program of China under grant 2023YFA1009002.

The research of GY was supported in part by NSFC grants 11771074 and National Key R\&D Program of China (No. 2020YFA0714102 and No. 2021YFA1003400).

\end{document}